\documentclass[12pt]{extarticle}
\usepackage{amsmath,amssymb, amsthm}
\usepackage{geometry,comment}
\usepackage{tikz-cd}
\newtheorem{theorem}{Theorem}[section]
\newtheorem{lemma}[theorem]{Lemma}
\newtheorem{corollary}[theorem]{Corollary}
\newtheorem{conjecture}[theorem]{Conjecture}
\newtheorem*{example}{Example}
\newtheorem{definition}[theorem]{Definition}
\newtheorem{proposition}[theorem]{Proposition}
\newtheorem{remark}[theorem]{Remark}

\def\mb{\mathbb}
\def\mc{\mathcal}

\def\L{\Lambda}
\def\k{\mb K}

\def\C{\mb C}

\def\Z{\mb Z}
\def\Q{\mb Q}
\def\ph{\varphi}
\def\N{\mc N}

\def\ard{\dashrightarrow}
\def\oL{\ol L}
\def\t{\times}
\def\cl{\colon}
\def\emb{\hookrightarrow}
\def\ol{\overline}
\def\wt{\widetilde}

\renewcommand{\P}{{\mb P}}
\def\wdw{\wedge\dots\wedge}
\def\PP{\mathbb P^1\times\mathbb P^1}
\def\bs{\backslash}
\def\Set{{\mc Set}}
\def\F{{\mc Fields}}
\def\B{{\mc B}}

\def\T{\mc T}

\def\W{\mc W}

\def \tW{\wt{\W}}
\def \wW{\tW}
\def \wL{\L_{\geq m-1}}
\def \CH{\mc{N}}
\def \wCH{\wt{\CH}_{\geq m-1}}
\def \G{\Gamma}
\def \wG{\Gamma_{\geq m-1}}
\def \M{\mc M}

\newcommand{\res}[2]{\left.{#1}\right|_{#2}}
\newcommand{\EP}[2]{{\Lambda^{#1}\k(#2)^\times}}

\newcommand{\SRL}{\mathrm{RecMaps}}
\newcommand{\Hom}{\mathrm{Hom}}

\DeclareMathOperator{\ts}{\partial}
\DeclareMathOperator{\ord}{ord}
\DeclareMathOperator{\mult}{mult}
\DeclareMathOperator{\val}{val}

\DeclareMathOperator{\im}{Im}
\DeclareMathOperator*{\colim}{colim}

\DeclareMathOperator{\spec}{Spec}
\DeclareMathOperator{\dval}{DiscVal}

\DeclareMathOperator{\Div}{Div}
\DeclareMathOperator{\supp}{supp}
\usepackage{csquotes }

\begin{document}
\title{On Goncharov's conjecture in next to Milnor degree}
\author{Vasily Bolbachan\footnote{This paper is an output of a research project implemented as part of the Basic Research Program at the National Research University Higher School of Economics (HSE University). This paper was also supported in part by the contest \textquote{Young Russian Mathematics}.}}

\maketitle
%\email{vbolbachan@gmail.com}
%\address{Skolkovo Institute of Science and Technology, Nobelya str., 3, Moscow, Russia, 121205; Faculty of Mathematics, National Research University Higher School of Economics, Russian Federation, Usacheva str., 6, Moscow, 119048; HSE-Skoltech International Laboratory of Representation
%Theory and Mathematical Physics, Usacheva str., 6, Moscow, 119048}

%\keywords{19E15, 11G55}

\begin{abstract}
   Let $\mathbb K$ be a field of characteristic zero. We prove that its motivic cohomology in degree $m-1$ and weight $m$ is rationally isomorphic to the cohomology of the polylogarithmic complex. This gives a partial extension of A. Suslin theorem describing the indecomposable $K_3$ of a field.
\end{abstract}


%\thanks{This paper is an output of a research project implemented as part of the Basic Research Program at the National Research University Higher School of Economics (HSE University). This paper was also supported in part by the contest "Young Russian Mathematics"}

\tableofcontents
\section{Introduction}
\label{sec:introduction}
Let $X$ be a smooth variety. S. Bloch  \cite{bloch1986algebraic} defined higher Chow group of $X$. These groups are bigraded and denoted by $CH^m_\Delta(X,n)$. They are defined as $n$-th homology of the complex $Z^m_\Delta(X,*)$, where $Z^m_\Delta(X,n)$ is defined as the group of codimension $m$ algebraic cycles on $n$-dimensional algebraic simplex over $X$ intersecting all the faces properly. Later A. Suslin and V. Voevodsky defined motivic cohomology $H^{j,m}(X,\Z)$ of $X$ and showed that $H^{j,m}(X,\Z)\cong CH^m_{\Delta}(X,2m-j)$ \cite{suslin2000higher, friedlander2002spectral, voevodsky2002motivic}. Let $H^{j,m}(X,\Q)=H^{j,m}(X,\Z)\otimes_\Z \Q$.

In this paper, we concentrate on the case when $X$ is a point. Let $\k$ be a field of characteristic zero. Everywhere we work over $\Q$. For instance, $\Lambda^m K^\times=\Lambda^m (K^\times\otimes_\Z\Q)$. 

In \cite{goncharov1995geometry}, A. Goncharov defined the polylogarithmic complex $\Gamma(\k,m)$. This complex looks as follows:
$$\Gamma(\k,m)\colon \mathcal B_m(\k)\xrightarrow{\delta_m} \mathcal B_{m-1}(\k)\otimes \k^\times\xrightarrow{\delta_m}\dots\xrightarrow{\delta_m}\mathcal B_2(\k)\otimes \Lambda^{m-2}\k^\times\xrightarrow{\delta_m}\Lambda^m \k^\times.$$

This complex is concentrated in degrees $[1,m]$. The group $\mathcal B_m(\k)$ is the quotient of the free abelian group generated by symbols $\{x\}_m, x\in \mathbb P^1(\k)$ by some explicitly defined subgroup $\mathcal R_m(\k)$ (see \cite{goncharov1994polylogarithms}).   The differential is defined as follows: $\delta_m(\{x\}_k\otimes x_{k+1}\wedge \dots \wedge x_m)=\{x\}_{k-1}\otimes x\wedge x_{k+1}\wedge \dots \wedge x_m$ for $k>2$ and $\delta_m(\{x\}_2\otimes x_3\wdw x_m)=x\wedge (1-x)\wedge x_3\wdw x_m$.

 As it was noted in Section 4.2 of \cite{goncharov1994polylogarithms} the group $\mc R_2(\k)$  is generated by the following elements:

$$\sum\limits_{i=1}^5(-1)^i\{c.r.(z_1,\dots, \widehat z_i,\dots, z_5)\}_2, \{0\}_2, \{1\}_2, \{\infty\}_2.$$
  In this formula $z_i$ are five different points on $\mathbb P^1$ and $c.r.(\cdot)$ is \emph{the cross ratio} defined by the formula $$c.r.(a,b,c,d)=\dfrac{(a-c)(b-d)}{(a-d)(b-c)}. $$ A. Goncharov formulated the following conjecture:

  \begin{conjecture}[Goncharov's conjecture]
  \label{conj:Goncharov_K_theory}
      For any field $\k$ there is a functorial isomorphism $H^{j,m}(\spec\k,\Q)\cong H^j(\Gamma(\k,m))$. 
  \end{conjecture}

  The conjecture \ref{conj:Goncharov_K_theory} was formulated  30 years ago, however, it was known only in a very few cases: when $j=m$ \cite{totaro_1992_milnork,nesterenko1990homology} and when $j=1, m=2$ \cite{suslin1991k3}. We remark that both of these cases were known before \cite{goncharov1995geometry}. The goal of this paper is to prove the following theorem:

  \begin{theorem}
\label{th:main_first_intro}
      Let $\k$ be a field of characteristic zero and $m\in\Z, m\geq 2$. We have the canonical isomorphism $H^{m-1,m}(\spec\k,\Q)\cong H^{m-1}(\Gamma(\k,m))$.
  \end{theorem}

  We believe that our result can be extended to arbitrary characteristics. However, for simplicity, we deal only with the case of characteristic zero.

  Let us give a more explicit version of this theorem. Denote by $\square^n=(\P^1\bs\{1\})^n$ the algebraic cube of dimension $n$. Subvarieties of $\square^n$ given by the equations of the form $z_i=0,\infty$ are called \emph{faces.} The higher Chow groups are defined as the cohomology of a certain complex $\mc Z_\square^m(\k,*)$. The group $\mc Z_\square^m(\k, n)$ is a quotient of the group of codimension $m$ algebraic cycles on $\square^{n}$ intersecting all faces properly by the subgroup of degenerate cycles \cite{bloch_rriz_1994_mixed}. The differential is given by the signed sum of restrictions to codimension one faces. We use the alternating version $\N(\k,m)_*$ of the complex $\mc Z_\square^m(\k,2m-*)$ which we define in Section \ref{sec:preliminary_def}. 
  
Let $Y$ be a variety and $f_1,\dots, f_n\in \k(Y)^\t$. Define a cycle $[f_1,\dots, f_n]_Y$ on $\square^n$ as follows. Consider a rational map $Y\ard \square^n$ given by $z\mapsto (f_1(z),\dots, f_n(z))$. Denote by $Z$ the closure of its image and by $\psi$ the natural map $Y\ard Z$. If $\dim Z< \dim Y$ we set $[f_1,\dots, f_n]_Y=0$. Otherwise, we set $[f_1,\dots, f_n]_Y=\deg(\psi)[Z]$. When the corresponding cycle intersects faces properly, this cycle gives the canonical element in the complex $\CH(\k,m)_j$, where $m=n-\dim Y, j=n-2\dim Y$. We denote this element by the same symbol.

For a chain complex $C_*$ and $k\in \Z$, denote by $\tau_{\geq k}C_*$ its \emph{canonical truncation}. By definition, if $j>k$ than $(\tau_{\geq k}C_*)_j=C_j$. If $j=k$ than $(\tau_{\geq k}C_*)_j=C_k/(\im d)$ and if $j<k$ than $(\tau_{\geq k}C_*)_j=0$. The differential of $\tau_{\geq k}C_*$ is induced by the differential of $C_*$. Similarly, denote by $\sigma_{\geq k}$ its \emph{stupid truncation}.  By definition, if $j\geq k$ than $(\sigma_{\geq k}C_*)_j=C_j$ and if $j<k$ than $(\sigma_{\geq k}C_*)_j=0$. The differential of $\sigma_{\geq k}C_*$ is induced by the differential of $C_*$.

Define a subcomplex $\M(\k,m)$ \cite{cubical1999herbert} of the complex $\tau_{\geq m-1}\CH(\k, m)$ as follows.  The vector space $\M(\k,m)_{m-1}$ is generated by the elements of the form $[f_1,f_2,\dots, f_{m+1}]_X$, where: $X$ is a complete curve, the divisors of the functions $f_1, f_2$ are disjoint and the functions $f_i,i\geq 3$ are constants. $\M(\k,m)_{m}$ is defined as $d(\M(\k,m)_{m-1})$. One can check that for any $g_1,g_2,c_2,\dots, c_{m}\in\k^\t$, the element
$$[g_1g_2,c_2,\dots, c_m]_{\spec \k}-[g_1,c_2,\dots, c_m]_{\spec \k}-[g_2,c_2,\dots, c_m]_{\spec \k}$$
belongs to $\M(\k,m)_m$.
We will show that $\M(\k,m)$ is acyclic. This essentially follows from Beilinson-Soule vanishing $H^{0,1}(\k,\Q)=0$ which is well-known \cite{bloch1986algebraic},\cite{voevodsky_six_lectures}.

Assume that $\k$ is algebraically closed. Denote by $V$ the $m$-th symmetric power of the complex $\N(\k,1)$. In the beginning of Section \ref{sec:W_isomorphsim} we will establish the following isomorphisms $$\mc M(\k,m)_{m-1}\cong\mathrm{coker}(d\colon V_{m-2}\to V_{m-1}),\quad \M(\k,m)_m\cong \im(d\colon V_{m-1}\to V_m).$$

Denote $\wCH(\k,m)=(\tau_{\geq m-1}\CH(\k,m))/\M(\k,m)$ and $\wG(\k,m)=\tau_{\geq m-1}\G(\k,m)$. Define a morphism of complexes $$\T\colon \wG(\k,m)\to \wCH(\k,m).$$  The chain complexes $\wG(\k,m),\wCH(\k,m)$ both concentrated in degrees $m-1$ and $m$. So, to define $\T$ it is enough to define it on the elements of the form $\{a\}_2\otimes c_3\wdw c_{m}$ and $c_1\wdw c_m$. The element $\{a\}_2\otimes c_3\wdw c_{m}$ goes to 
$[t,(1-t),(1-a/t),c_3,\dots, c_{m}]_{\P^1}.$ The element $c_1\wdw c_m$ goes to
$[c_1,\dots, c_m]_{\spec\k}.$ It follows easily from \cite{cubical1999herbert}, that $\T$ is a well-defined morphism of complexes. We will give a detailed proof in Section \ref{sec:Totaro_map}. Here is a more precise version of our main result:

\begin{theorem}
\label{th:main_second_intro}
    Let $\k$ be a field of characteristic zero and $m\geq 2$. The map $\T$ is a quasi-isomorphism. When $\k$ is algebraically closed, this map is an isomorphism.
\end{theorem}

%\begin{remark}
 %   In the case $m=2$, it follows from Suslin theorem \cite{suslin1991k3} that $H^1_{mot}(\k,\Q(2))$ is canonically isomorphic to $H^1(\Gamma(\k,2))$. However, it is not clear whether the following diagram commutes:
  %  \begin{equation}
    %\label{diagram:Suslin_and_Totaro}
     %   \begin{tikzcd}
      %       K_3^{ind}(\k)\ar[r]\ar[d]& H^1(\Gamma(\k,2))\ar[d, equal]\\
       %      H^1(\wt{\CH}_{\geq 1}(\k,2))\ar[r,"\T"] & H^1(\Gamma(\k,2))
        %\end{tikzcd}
        %\end{equation}
    %In this diagram the top horizontal arrow comes from Suslin theorem, left vertical comes from isomorphism between higher Chow groups and algebraic $K$-theory \cite{bloch1986algebraic,levine1994bloch} and the lower horizontal arrow comes from our main result.
    %So, even in the case $m=2$, the fact that $\T$ is a quasi-isomorphism is a new result.
%\end{remark}

%\begin{remark}
 %   To prove our main result we rely on \cite{rudenko2021strong} where it was constructed the norm map on complex $\wG(\k,m)$. This result, in turn, relies on the main result of \cite{suslin1991k3}. So in the case $m=2$ we do not give a new proof of the fact that  $H^1_{mot}(\k,\Q(2))$ is canonically isomorphic to $H^1(\Gamma(\k,2))$. However when $\k$ is algebraically closed, it is possible to modify our arguments to give such a proof.
%\end{remark}

\subsection{The idea of the proof}
According to \cite{rudenko2021strong,bloch1986algebraic} on both complexes there is a norm map with good properties. This reduces the general case to the case when $\k$ is algebraically closed. 

To prove the main result we will construct some intermediate complex $\L(\k,m)$ and a morphism of complexes $W\cl \N(\k,m)\to \L(\k,m)$. Let $\wL(\k,m)$ be $\tau_{\geq m-1}\L(\k,m)$. Denote by $\wt{\mc W}$ the map $\wCH(\k,m)\to \wL(\k,m)$ induced by $\mc W$. Let $\mc T'=\wt{\mc W}\circ\mc T$. We will show that all maps in the following diagram are isomorphisms:

\begin{equation}
    \label{diagram:main_triangle}
        \begin{tikzcd}
             & \wL(\k,m)\ar[dl,leftarrow,"\T'"'] \ar[dr,leftarrow,"\tW"] &\\
             \wG(\k,m)\ar[rr,"\T"] &   & \wCH(\k,m)
        \end{tikzcd}
        \end{equation}

Let $X$ be a proper curve and $g_1,g_2,f_2\dots,f_{m+1}$ are rational functions on $X$. Let $g_3=g_1g_2$. It was stated in \cite{cubical1999herbert} and we will prove in Subsection 4.1, that if for any $i$ the cycle $[g_i,f_2,\dots, f_{m+1}]_X$ intersects faces properly, we get the following relation in the group $\wCH(\k,m)$:
\begin{equation}
\label{formula:multiplicative_relation}
    [g_3,f_2,\dots, f_{m+1}]_X=[g_1,f_2,\dots, f_{m+1}]_X+[g_2,f_2,\dots, f_{m+1}]_X.
\end{equation}

However, when these cycles are not admissible, this relation does not make any sense. This suggests that there is some chain complex (we call it $\L(\k,m)$) with the following properties:
\begin{enumerate}
    \item The complex $\L(\k,m)$ is a well-defined chain complex concentrated in degrees $m-2,m-1$ and $m$.
    \item Let $Y$ be a variety of dimension $\leq 2$. Any sequence $f_1,\dots, f_n$ of non-zero rational functions on $Y$ gives a well-defined element in $\L(\k,m)_{2m-n}$, where $m=n-\dim Y$. Call it $\omega(Y;f_1,\dots, f_n)$. The assignment $(f_1,\dots, f_n)\mapsto \omega(Y;f_1,\dots, f_n)$ gives a well-defined homomorphism
    $$\k(Y)^\times\wedge\dots\wedge \k(Y)^\times\to \L(\k,m)_{2m-n}.$$
    \item When the cycle $[f_1,\dots, f_n]_Y$ intersects faces properly, the differential of $\omega(Y;f_1,\dots, f_n)$ is closely related to the differential of the cycle $[f_1,\dots, f_n]_Y$ in $\mc N(\k,m)$.
\end{enumerate}

 Let us make this idea precise. We recall that \emph{an alteration} is a proper surjective morphism which is generically finite \cite[Introduction]{de1996smoothness}.

    \begin{definition} 
    Let $j\in \{m-2,m-1,m\}.$ Denote by $\wt\L(\k, m)_j$  a vector space freely generated by the isomorphism classes of the pairs $(Y, a)$, where: $Y$ is a complete variety over $\k$ of dimension $m-j$ and $a\in\Lambda^{2m-j}(\k(Y)^\t)$. Denote by $[Y,a]\in\wt\L(\k,m)_j$ the corresponding element.
Denote by $\L(\k, m)_j$ a quotient of $\wt\L(\k, m)_j$ by the following relations:
$$[\wt Y,\ph^*(a)]=(\deg\ph)[Y,a].$$
$$[Y,a+b]=[Y, a] + [Y, b].$$

In this formula $\ph\colon \wt Y\to Y$ is any alteration. The map $\ph^*$ is defined by the formula
$\ph^*(\alpha_1\wdw \alpha_n)=\ph^*(\alpha_1)\wdw \ph^*(\alpha_n)$. We remark that we write the group law in the group $\k(Y)^\t$ additively. %So in the case $2m-j=1$ the addition in this group corresponds to \emph{the multiplication} of the functions.
\end{definition}

We recall that we assume that $\k$ is algebraically closed. In the case $j=m$ we get $\Lambda(\k,m)_m=\Lambda^m\k^\times$. Denote by $\Lambda(
\k,m
)$ the following cochain complex:
$$\Lambda(\k,m)\colon\quad \Lambda(\k,m)_{m-2}\to \Lambda(\k,m)_{m-1}\to \Lambda^m\k^\times.$$

The differential is given by a so-called tame symbol map defined by A. Goncharov \cite{goncharov1995geometry}.
In Section \ref{sec:Lambda_properties} we will give a precise definition of this complex and will show that it is well-defined.  Besides, we will construct a morphism of complexes $$\W\colon \sigma_{\geq m-2}\CH(\k,m)\to \L(\k,m),$$
where $\sigma_{\geq m-2}$ is the stupid truncation. The cycle $[f_1,\dots, f_n]_Y$ corresponds to $[\overline Y, f_1\wdw f_n]$ where $\ol Y$ is the closure of $Y$ in $(\P^1)^{2m-j}$.

Let us emphasize that there are two main differences between $\N(\k,m)$ and $\L(\k,m)$:
\begin{enumerate}
    \item Let $Y$ be a proper variety. For arbitrary non-zero functions $f_1,\dots, f_n$ on $Y$, the cycle $[f_1,\dots, f_n]_Y$, as a cycle on $\square^n$, may not intersect faces properly. In this case, this cycle does not give any element in the Bloch complex.  However, the element $[Y,f_1\wdw f_n]$ is a well-defined element in the complex $\L(\k,n-2\dim Y)_{n-\dim Y}$.
    \item For arbitrary non-zero rational functions $g_1,g_2,f_2,\dots, f_n$ we have
    $$[Y, (g_1g_2)\wedge f_2\wdw f_n]=[Y, g_1\wedge f_2\wdw f_n]+[Y, g_2\wedge f_2\wdw f_n].$$
    The corresponding identity in $\N(\k,m)$, if true, is satisfied only up to a coboundary.
\end{enumerate}

%\begin{conjecture}
%    The map $\W$ is a quasi-isomorphism.
%\end{conjecture}

%We are not able to prove this conjecture in full generality. However, we will prove it in the particular case which is enough to prove our main result. Set $\wL(\k,m)=\tau_{\geq m-1}\L(\k,m)$. The map $\W$ induces a map $\wCH(\k,m)\to \wL(\k,m)$ which we denote by $\tW$.

\begin{theorem}
\label{th:W_isomorphism_intro}
    Let $\k$ be an algebraically closed field of characteristic zero. The map $\wt {\mc W}$ induces the canonical isomorphism $$\wt{\mc N}_{\geq m-1}(\k,m)\to \Lambda_{\geq m-1}(\k,m).$$
\end{theorem}

For a divisor $D=\sum\limits_{\alpha}n_\alpha[D_\alpha]$ on $X$, denote by $\supp(D)$ \emph{its support} defined by the formula $\supp D=\sum\limits_{\alpha} [D_\alpha]$. For some number of divisors $D_1,\dots, D_n$ define $$\supp(D_1,\dots, D_n)=\supp(\supp(D_1)+\dots+\supp(D_n)).$$

Theorem \ref{th:W_isomorphism_intro} will be proved in Section \ref{sec:W_isomorphsim}. The idea is contained in the following statement:

\begin{proposition}
\label{pr:weakly_adm_is_str_adm_intro}
    Let $X$ be a smooth variety. Let $f_1,\dots, f_n\in\k(X)$. Assume that $$\supp((f_1),\dots, (f_n))$$ is a snc divisor. Then $f_1\wdw f_n$ is a linear combination of elements of the form $g_1\wdw g_n$ such that 1) $\supp((g_1),\dots, (g_n))$ is a snc divisor and 2) the divisors $(g_1),\dots, (g_n)$ have no common components.
\end{proposition}

\begin{example}
    Let $X=\P^1$ and $f_1=x-a, f_2=x-b, a,b\in \k, a\ne b$. The divisors $(f_1), (f_2)$ have a common point, namely $\infty$. Take $c\in \k, c\not\in\{a,b\}$. We can write:
    $$(x-a)\wedge (x-b)=1/2\left(\dfrac{x-a}{x-c}\wedge (x-b)+(x-a)\wedge \dfrac{x-b}{x-c}+\dfrac{x-a}{x-b}\wedge (x-c)\right).$$
    It is easy to see that, for example, the divisors of the functions $(x-a)/(x-c), (x-b)$ have no common components. This example was discovered by Eric Wofsey \cite{4728781} as an answer to the author's question on the site math.stackexchange. 
\end{example}

This proposition easily implies that the map $\wW$ is surjective. The injectivity will be proved in Section \ref{sec:W_isomorphsim}.

As it was stated above, the map $\tW$ is an isomorphism. It remains to prove the following: 

\begin{theorem}
\label{th:Totaro_isomorphism_intro}
    Let $\k$ be algebraically closed. The map $\T'$ is an isomorphism.
\end{theorem}

%\begin{remark}
 %   The proof of the main result from \cite{cubical1999herbert} is quite computational. It turns out that we can prove our main result without using \cite{cubical1999herbert}. In Appendix \ref{sec:Totaro_is_well_defined} we will give a direct proof of the fact that $\mc T'$ is well-defined. So it is enough to prove that $\mc T'$ and $\wW$ are quasi-isomorphism. One can show that $H^i(\wL(\k,m))$ satisfies Galois descent. So we can assume $\k$ to be algebraically closed. In this case these statements follow from Theorem \ref{th:W_isomorphism_intro} and Theorem \ref{th:Totaro_isomorphism_intro}.
%\end{remark}

This theorem will be proved in Section \ref{sec:Lambda_and_polylogarithms}. The injectivity of the map $\T'$ will follow from some generalization of the main result of \cite{bolbachan_2023_chow} which we will prove in Section \ref{sec:strong_suslin_reciprocity_law}. The fact that $\mc T'$ is surjective easily follows from the following proposition:

\begin{proposition}
\label{pr_intro:Lambda_is_generated_by_rationally_connected}
    Let $\k$ be algebraically closed and $m\in\Z, m\geq 2$. The group $\L(\k,m)_{m-1}/\im(d)$ is generated by the elements of the form $[\P^1, a]$, where  $a\in\L^{m+1}\k(\P^1)^\t$.
\end{proposition}

\subsection{The structure of the paper} In Section \ref{sec:preliminary_def} we state our main definitions. In Section \ref{sec:Lambda_properties} we will define the complex $\L(\k,m)$ and will study its properties. In Section \ref{sec:W_isomorphsim} we will prove that the map $\wW$ is an isomorphism (for algebraically closed $\k$). In Section \ref{sec:strong_suslin_reciprocity_law} we will prove some generalization of the main result of \cite{bolbachan_2023_chow}. Finally, in Section \ref{sec:Lambda_and_polylogarithms} we will prove our main result.

\subsection{Acknowledgements}
The author is grateful to A.Levin, D. Rudenko, E. Wofsey for useful discussions. 

\section{List of notations}
\begin{itemize}
\item $m$ -- the motivic weight
\item $j$ -- the degree in motivic cohomology
\item $n$  -- index in $K$-theory. We have $$K_n^{(m)}(\k)\cong CH^m(\k,n)\cong H^{j,m}(\k,\Q).$$
\item $B_2(F)$ -- the second Bloch group of the field $F$.
\item $\tau_{\geq m-1},\sigma_{\geq m-2}$ -- the canonical and the stupid truncation, respectively.
    \item $\Gamma(F,m)$ -- weight $m$ polylogarithmic complex of the field $F$.
    \item $\ts_D$ -- tame-symbol map $\Lambda^n\k(X)\to \Lambda^{n-1}\k(D)^\times$ defined in the beginning of Section \ref{sec:preliminary_def}.
    \item $\mc Z_\square^m(\k,*)$ -- the chain complex calculating cubical higher Chow groups. 
    \item $\mc N(\k,m)$ -- the alternating version of the complex $\mc Z_\square^m(\k,2m-*)$.
        
    \item $\Gamma_{\geq m-1}(F,m), \mc N_{\geq m-1}(\k,m)$ -- the canonical truncation of the complexes $\Gamma(F,m)$ and $\mc N(\k,m)$
    \item $\mc M(\k,m)$ -- a subcomplex of $\mc N_{\geq m-1}(\k,m)$ defined in the beginning of Section \ref{sec:W_isomorphsim}.
    \item $\wt {\mc N}_{\geq m-1}(\k,m)$ -- the quotient complex $\mc N_{\geq m-1}(\k,m)/\mc M(\k,m)$
\item $\Lambda(\k,m)$ -- the chain complex defined in Section \ref{sec:Lambda_properties}.
    \item $\Lambda_{\geq m-1}(\k,m)$ -- its canonical truncation
    \item $\mc W$ -- the natural $\sigma_{\geq m-2}\mc N(\k,m)\to \L(\k,m)$ defined in Section \ref{subsec:lambda_higher_chow_groups}. In this formula $\sigma_{\geq m-2}$ is the stupid truncation.

\item $\wt{\mc W}$ -- the corresponding map $\wt {\mc N}_{\geq m-1}(\k,m)\to \Lambda_{\geq m-1}(\k,m)$.
\item $\mc T$ -- the natural map $\Gamma_{\geq m-1}(\k,m)\to \wt {\mc N}_{\geq m-1}(\k,m)$ defined before statement of Theorem \ref{th:main_first_intro}.
\item $\mc T'$ -- the composition $\wt{\mc W}\circ \mc T$. See Section \ref{subsec:map:Totaro_prime} for an alternative definition.
    \item $[X,a]$ -- an element of the group $\L(\k,m)_j$. In this formula $X$ is a proper $m-j$ dimensional variety and $a\in\Lambda^{2m-j}\k(X)^\times$.
    \item $[f_1,\dots, f_n]_X$ -- an algebraic cycle on $(\mb P^1\bs\{1\})^n$ parametrized by $X$ via the map $$z\mapsto (f_1(z),\dots, f_n(z)).$$ If this cycle intersects faces properly, we denote by the same symbol the corresponding element in the groups $\mc Z_\square^m(\k,2m-j)$ and $\mc N(\k,m)_j$.
    
\end{itemize}

\section{Preliminary definitions and
lemmas}
\label{sec:preliminary_def}
We recall that we work over $\mathbb Q$. Throughout the article we assume that the field $\k$ has characteristic zero. \emph{An alteration} $\ph\colon \wt Y\to Y$ is a proper surjective morphism between integral schemes which is generically finite. 

Let $(F,\nu)$ be a discrete valuation field. Denote $\mc O_\nu=\{x\in F|\nu(x)\geq 0\}$ the valuation ring, $m_\nu=\{x\in F|\nu(x)>0\}$  the maximal ideal and $\overline F_\nu =\mc O_\nu/m_\nu$ the residue field. We recall that an element $a\in F^\t$ is called \emph{a uniformiser} if $\nu(a)=1$ and \emph{a unit} if $\nu(a)=0$. For $u\in \mc O_\nu$ denote by $\overline u$ its residue class in $\ol F_\nu$.

The proof of the following proposition can be found in \cite{goncharov1995geometry}:

\begin{proposition}
\label{prop:tame_symbol_classiacal}
Let $(F,\nu)$ be a discrete valuation field and $n\geq 1$. There is a unique map $\ts_\nu\colon \L^n(F^\t)\to \L^{n-1}\ol F_\nu^\t$ satisfying the following conditions:

\begin{enumerate}
    %\item For any units $u_1,\dots u_n$ we have $\ts[n]_\nu(u_1\wedge\dots \wedge u_n)=0$.
    \item When $n=1$ we have $\ts_{\nu}(a)=\nu(a)$.
   \item For any units $u_2,\dots, u_n\in F$ we have $\ts_\nu(x\wedge u_2\wdw u_n)=\nu(x)\cdot( \ol{u_2} \dots \wedge \overline{u_n})$.

\end{enumerate}
\end{proposition}

The map $\ts_\nu$ is called \emph{tame symbol map}. The following lemma is straightforward.

\begin{lemma}
\label{lemma:leibniz_rule_tame_symbol}
Let $(F, \nu)$ be a discrete valuation field. Let $k, n$ be two natural numbers satisfying the condition $k<n$. Let $a_1,\dots, a_k\in F^\t$ and $a_{k+1},\dots, a_n$ be $\nu$-adic units. Then:
$$\ts_\nu(a_1\wedge\dots\wedge a_n)=\ts_\nu(a_1\wedge\dots\wedge a_k)\wedge \overline{a_{k+1}}\wedge \dots\wedge \overline{a_n}.$$
\end{lemma}

Let $X$ be a smooth algebraic variety and $D\subset X$ be a prime divisor. Denote by $\nu_D$ a discrete valuation corresponding to $D$. Abusing notation, we write $\ts_D$ for $\ts_{\nu_D}$. Let $Y$ be a variety. Denote $\EP{n}{Y}:=\L^n(\k(Y)^\t\otimes_\Z\Q)$. Here $\k(Y)^\t$ is the multiplicative group of the field $\k(Y)$ considered as an abelian group. Let $\ph\colon Y_1\to Y_2$ be a dominant morphism between varieties. Denote by $\ph^*$ the map $\k(Y_2)^\t\to \k(Y_1)^\t$ given by the formula $\ph^*(\alpha)=\alpha\circ \ph$. Denote by the same symbol the map $\EP{n}{Y_2}\to \EP{n}{Y_1}$ given by the formula $\ph^*(\alpha_1\wdw \alpha_n)=\ph^*(\alpha_1)\wdw \ph^*(\alpha_n)$. Let $D_1\subset Y_1, D_2\subset Y_2$ be prime divisors such that $\ol{\varphi(D_1)}= D_2$. Denote by $e(D_1, D_2)$ the multiplicity of $\ph^{-1}(D_2)$ along $D_1$. This number can be computed as $\ord_{D_1}(\ph^*(\alpha))$, where $\alpha$ is any rational function on $Y_2$ satisfying $\ord_{D_2}(f)=1$. 

\begin{lemma}
\label{lemma:functoriality_of_residue_2}
    Let $\ph\colon Y_1\to Y_2$ be a dominant morphism between smooth varieties. Let $D_1\subset Y_1, D_2\subset Y_2$ be two prime divisors such that $\ol{\ph(D_1)}=D_2$. Denote by $\ol\ph $ the natural map $D_1\to D_2$. Then for any $a\in\EP{n}{Y}$ we have:
    $$\ts_{D_1}(\ph^*(a))=e(D_1, D_2)(\ol\ph)^*(\ts_{D_2}(a)).$$
\end{lemma}
\begin{proof}
    We can assume that $a=\xi_1\wdw \xi_n$ such that $\ord_{D_2}{\xi_1}=1$ and $\ord_{D_2}(\xi_i)=0$ for any $i>1$. The number $\ord_{D_1}(\ph^*(\xi_i))$ is equal to $e(D_1, D_2)$ if $i=1$ and is equal to $0$ otherwise. Now the statement follows from Lemma \ref{lemma:leibniz_rule_tame_symbol}.
\end{proof}
Let  $\ph\cl Y_1\to Y_2$ be a morphism of finite degree and fix two prime divisors $D_1\subset Y_1$ and $D_2\subset Y_2$ such that $\ol{\ph(D_1)}= D_2$. We have a finite extension $\k(D_2)\subset \k(D_1)$. Denote its degree by $f(D_1, D_2)$. Denote by $Div(\ph)_{cont}$ the set of divisors contracted under $\ph$. Let $D\subset Y_2$ be a prime divisor. Denote by $Div(\ph, D)$ the set of prime divisors $D'\subset Y_1$ such that $\ol{\ph(D')}=D$. There is a bijection between the set $Div(\ph, D)$ and the set of extensions of the discrete valuation $\nu_D$ to the field $\k(Y_1)$. The following statement is well known:

\begin{lemma}
\label{lemma:degree_is_a_sum_multiplicties}
    Let $\ph\colon Y_1\to Y_2$ be an alteration between smooth varieties. Let $D\subset Y_2$ be a prime divisor. The following formula holds:  
    $$\deg \varphi = \sum\limits_{D'\in Div(\ph, D)}e(D', D)f(D', D).$$
\end{lemma}
\subsection{Strictly regular elements}
Let us recall the definition of a simple normal crossing divisor.
Let $X$ be a smooth variety. For a divisor $D=\sum\limits_{\alpha}n_\alpha[D_\alpha]$ on $X$, denote by $\supp(D)$ \emph{its support} defined by the formula $\supp D=\sum\limits_{\alpha} [D_\alpha]$.  For some number of divisors $D_1,\dots, D_n$, we set $$\supp(D_1,\dots, D_n)=\supp(\supp(D_1)+\dots+\supp(D_n)).$$ 

A divisor $D=\sum n_\alpha D_\alpha$ is called \emph{ simple normal crossing} if for any closed point $z\in D$,  there is a regular local system of parameters $\pi_1,\dots, \pi_d$ at $z$, such that locally $D$ is given by the equation $\pi_1\dots \pi_l=0$ for some $l\leq d$.
A divisor $D$ is called \emph{supported on a simple normal crossing divisor} if $\supp(D)$ is a simple normal crossing divisor.  %A divisor $D$ is called \emph{supported on a simple normal crossing divisor locally at $x$}, if the restriction of this divisor to some open neighborhood of the point $x$ is supported on a simple normal crossing divisor.

%Let $X$ be a variety and $a\in \Lambda^n \k(X)^\times$. We can write
%$$a=\sum_\alpha l_\alpha f_1^\alpha\wdw f_n^\alpha.$$

%Consider the following divisor:
%$$D=\sup\left(\sum_{i,\alpha}\sup((f_i^\alpha))\right).$$

%Quite often, it is natural to suppose that $D$ is supported on a simple normal crossing divisor. However, it turns out, that the following definition is more convenient:

\begin{definition}
\label{def:strictly_regular}
    Let $X$ be a smooth variety and $x\in X$. An element $b\in\EP{n}{X}$ is  called \emph{strictly regular at $x$} if there is some open subset $U$ containing $x$ and a presentation of $b$ in the form
    $$\sum_\alpha l_\alpha f_1^{\alpha}\wdw f_n^{\alpha}$$
    such that for any $\alpha$, the restriction of the divisor $\supp((f_1^\alpha),\dots, (f_n^\alpha))$ to $U$ is supported on a simple normal crossing divisor on $U$.  The element $b$ is called \emph{strictly regular} if this condition holds for any $x\in X$.

    More generally, an element $b$ of the group $\Gamma(\k(X),m)_j$ is called \emph{strictly regular} at $x$, if there is some open subset $U$ containing $x$ and a presentation of $b$ in the form $\sum_\alpha l_\alpha\{f_1^\alpha\}_k\otimes f_2^{\alpha}\wdw f_r^{\alpha}$ such that for any $\alpha$ the restriction of the  divisor $\supp((f_1^\alpha),\dots, (f_r^{\alpha}))$ to $U$ is supported on a simple normal crossing divisor on $U$. The element $b$ is called \emph{strictly regular} if this condition holds for any $x\in X$.

    %The element $b$ is called \emph{strictly regular} if this condition holds for any $x\in X$.
\end{definition}

We have the following simple lemma:

\begin{lemma}
\label{lemma:strictly_regular_blow_up}
    Let $X$ be a proper variety and $a\in \Lambda^n\k(X)^\times$. There is a proper birational morphism $\varphi\colon \wt X\to X$ such that $\varphi^*(a)\in \Lambda^n\k(\wt X)^\times$ is strictly regular.
\end{lemma}

\begin{proof}
    Indeed, let 
    $$a=\sum_{\alpha}l_\alpha f_1^\alpha\wdw f_n^\alpha$$
    and
    $$D=\sup\left(\sum_{i,\alpha}\sup((f_i^\alpha))\right).$$
    By Hironaka resolution of singularities \cite{hironaka_1}, there is a proper birational morphism $\wt X\to X$ such that $X$ is smooth and $\varphi^*(D)$ is supported on a snc divisor. It follows that $\varphi^*(a)$ is strictly regular.
\end{proof}

We will need the following simple lemma:

\begin{lemma}
\label{lemma:characterisation_of_strictly_regular_elements}
Let $X$ be a smooth $d$-dimensional algebraic variety and $x\in X$. Let $a\in\Gamma(k(X),m)_{j}$ and assume that $a$ is strictly regular at $x$. The element $a$ can be represented as a linear combination of elements of the following form:
    $$\{\pi_1^{n_1}\dots \pi_d^{n_d}u_0\}\otimes \pi_1\wdw \pi_{r'}\wedge u_{r'+1}\wdw u_r.$$
    Here $\pi_i$ is a regular system of parameters at x and all the functions $u_i$ are regular at $x$ and take non-zero values at this point.

\end{lemma}

\begin{proof}
The statement follows from the following fact, which in turn follows from the definition of a snc divisor. Let $f_1, \dots, f_k$ be non-zero rational functions on $X$ and $U$ be some open subset containing $x$. Assume that the restriction of the divisor $\sup((f_1),\dots, (f_k))$ to $U$ is supported on a simple crossing divisor on $U$.  Then  there is a regular system of parameters $\pi_1,\dots \pi_d$ at $x$, rational functions $u_1,\dots, u_k$ which are regular at $x$ and take non-zero values at this point and integers $l_{i,j}$ such that for any $1\leq i\leq k$, we have: $f_i=u_i\pi_1^{l_{i,1}}\dots \pi_d^{l_{i,d}}$. 
\end{proof}

\subsection{Higher Chow groups}
\label{subsec:higher_chow_groups}
We recall that two irreducible subvarieties $V,W$ of some ambient variety $Y$  \emph{intersect properly}, if for any irreducible component $Z$ of $V\cap W$, we have $\dim Z=\dim V+\dim W-\dim Y$.

The definition of Bloch's higher Chow group can be found in \cite{bloch_rriz_1994_mixed}. We use the cubical alternating version.  Denote by $\square^n=(\P^1\bs\{1\})^n$ the algebraic cube of dimension $n$. Subvarieties of $\square^n$ given by the equations of the form $z_i=0,\infty$ are called faces. A cycle is called \emph{admissible} if it intersects all the faces properly. Let $z^m_\square(\k,n)$ be the free abelian group generated by codimension $m$ admissible cycles on $\square^n$. We have the natural projections $\pi_{i,n}\colon \square^n\to\square^{n-1}$ given by the formula $\pi_{i,n}(t_1,\dots, t_i,\dots, t_n)=(t_1,\dots\hat{t_i}, \dots, t_n)$. Denote by $D^m(\k, n)\subset z^m(\k,n)$ a subgroup generated by the images of the maps  $\pi_{i,n}^*\colon z_\square^{m}(\k, n-1)\to z^m_\square(\k, n), i=1,\dots, n$. We set $\mc Z^m_\square(\k,n)= z^m_\square(\k,n)/D^m(\k,n)$.

If $X$ is a variety, $Z\subset X$ be an equidimensional closed subset and $Z'\subset Z$ be some irreducible component of $Z$, we denote by $mult_{Z'}(Z)$ \emph{the  multiplicity} of $Z'$ in $Z$(see \cite[1.5]{fulton2013intersection}). \emph{The cycle associated to $Z$} is defined by the formula
$$[Z]=\sum\limits_{Z'}\mult_{Z'}(Z)[Z'].$$
In this formula the sum is taken over all irreducible components of $Z$.

Let $i\in\{1,\dots, n\}$ and $\varepsilon\in \{0,\infty\}$. Denote by $F_{i,\varepsilon}$ a face given by the equation $t_i=\varepsilon$. Define a map $\ts_{i,\varepsilon}\colon \mc Z_\square^m(\k,n)\to \mc Z_\square^m(\k,n-1)$ given by the formula $[Z]\mapsto [Z\cap F_{i,\varepsilon}].$ In this formula we identify $F_{i,\varepsilon}$ with $\square^{n-1}$ and consider $[Z\cap F_{i,\varepsilon}]$ as a cycle on $\square^{n-1}$.

Define the differential by the following formula

$$d=\sum\limits_{i=1}^n(-1)^{i+1}(\ts_{i,0}-\ts_{i,\infty}).$$

One can show that $d^2=0$. Up to a sign this definition coincides with the definition given in \cite{bloch_rriz_1994_mixed}.

The main idea of this paper is to compare higher Chow groups with the complex $\L(\k,m)$. The generators of the complex $\L(\k,m)$ have the form $[X,f_1\wedge\dots\wedge f_n]$. These elements are clearly antisymmetric with respect to $f_i$. For this reason, it is convenient to use antisymmetric version of higher Chow groups, which we want to recall.

Let $\mb Q[S_n]$ be the group algebra of symmetric group. Define 
$$alt_{n}=\dfrac 1{n!}\sum\limits_{\sigma\in S_n}sgn(\sigma)[\sigma]\in\Q[S_n].$$

%Denote by $G_n$ the wreath product of $\Z/2\Z$ and $S_n$. The group $G_n$ has a natural $1$-dimensional representation $\rho_n$ uniquely defined by the property that its restriction to $S_n$ and to each $\Z/2\Z$ are non-trivial.
%Let $Alt_n\in\Q[G_n]$ be the corresponding projector, which is defined as

%$$Alt_{n}=\dfrac 1{|G_n|}\sum\limits_{g\in G_n}\rho_n(g)[g].$$

There is a natural action of the group $S_{n}$ on $\square^n$ by permuting coordinates.  Let $\N(\k,m)_j=alt_{2m-j}(\mc Z^m_\square(\k,2m-j))$. We consider $\N(\k,m)_*$ as a cohomological complex. It is easy to see that $\CH(\k,m)_*$ is a subcomplex of $\mc Z_\square^m(\k, 2m-*)$. Similarly to \cite[Proposition 5.1.]{bloch_rriz_1994_mixed}(see also p. 581 from loc.cit.) it can be shown that the natural inclusion is a quasi-isomorphism.

% The symmetric group permutes coordinates and the generator of the $i$-th $\Z/2\Z$ acts by the formula $(t_1,\dots, t_n)\mapsto (x, t_1,\dots, t_{i-1}, t_i^{-1}, t_{i+1},\dots, t_n)$.

\section{The complex $\L(\k,m)$}
\label{sec:Lambda_properties}
In this section, we assume that $\k$ is an algebraically closed field of characteristic zero.

\subsection{Definition of the complex $\Lambda(\k, m)$}
\label{sub:sec:Lambda_definition}
\begin{definition}
    Let $m\geq 0$ and $j\in\{m-2,m-1,m\}$. Set 
    \begin{equation*}
        \begin{cases}
            p=m-j\\ n=2m-j
        \end{cases}
    \end{equation*}
We have
\begin{equation*}
 \begin{cases}
            m=n-p\\ j=n-2p
        \end{cases}
        \end{equation*}
    Throughout the paper we will use the pairs of integers $(m,j), (n, p)$ interchangeably.
    \end{definition}
    \begin{definition} 
    Denote by $\wt\L(\k, m)_j$ a vector space freely generated by  the isomorphism classes of the pairs $(Y, a)$, where: $Y$ is a complete variety over $\k$ of dimension $p=m-j$ and $a\in\Lambda^n(\k(Y)^\t\otimes_\Z\Q)$. Denote by $[Y,a]\in\wt\L(\k,m)_j$ the corresponding element.
Denote by $\L(\k, m)_j$ the quotient of $\wt\L(\k, m)_j$ by the following relations:
$$[\wt Y,\ph^*(a)]=(\deg\ph)[Y,a].$$
$$[Y,a+b]=[Y, a] + [Y, b].$$

In this formula $\ph\colon \wt Y\to Y$ is any alteration. The map $\ph^*$ is defined by the formula
$\ph^*(\alpha_1\wdw \alpha_n)=\ph^*(\alpha_1)\wdw \ph^*(\alpha_n)$. 
\end{definition}
Let us formulate this definition in a more categorical fashion. Denote by ${\mc Var}_{p}$ the category of complete varieties over $\k$ of dimension $p$ and their proper dominant morphisms. Denote by ${\mc Vect}_\Q$ the category of vector spaces over $\Q$. Define a contravariant functor $$\mathrm{EP}^n\cl {\mc Var}_{p}\to {\mc Vect}_\Q.$$ Its value on some variety $Y\in {\mc Var}_{p}$ is equal to the vector space $\EP{n}{Y}=\L^n(\k(Y)^\t\otimes_{\mb Z}{\mb Q})$. For a morphism $\varphi\colon Y_1\to Y_2$ the corresponding map $\EP{n}{Y_2}\to \EP{n}{Y_1}$ is defined by the formula $\Lambda^n(\varphi)(a)=\dfrac 1{\deg{\varphi}}(\ph^*(a))$.

\begin{definition}
    $$\L(\k, m)_j=\colim\limits_{Y\in {\mc Var}_{p}}\EP{n}{Y}.$$
\end{definition}

As the category ${\mc Var}_{p}$ is essentially small this colimit is a well-defined vector space over $\Q$. It is easy to see that the two definitions above are equivalent.

For a variety $Y$ and $a\in \EP{n}{Y}$, we denote the corresponding element in $\L(\k, m)_j$ by $[Y, a]$. For any alteration $\ph\cl Y_1\to Y_2$ we have $[Y_2, a]=\dfrac 1{\deg \ph}[Y_1, \ph^*(a)]$.

The complex $\L(\k,m)$ concentrated in degrees $m-2,m-1,m$ and has the following form:

$$\L(\k,m)_{m-2}\to \L(\k,m)_{m-1}\to \L(\k,m)_m.$$

Let $j\in\{m-2,m-1\}$. Define a map $d\colon \L(\k,m)_j\to \L(\k,m)_{j+1}$ as follows. Let $[Y, a]\in \L(\k,m)_j$. Choose a proper birational morphism $\ph \colon \wt Y\to Y$ with smooth $\wt Y$ such that the element $\ph^*(a)$ would be strictly regular. (In the case $j=m-1$ the condition that $\ph^*(a)$ should be strictly regular is trivial). Such a morphism exists by Lemma \ref{lemma:strictly_regular_blow_up}.  Define
$$d([Y, a])=\sum\limits_{D\subset \wt Y}[D, \ts_D(\ph^*(a))].$$

In this formula the sum is taken over all prime divisors $D\subset \wt Y$. In this section, we will show that the map $d$ is well-defined and satisfies $d^2=0$.

\subsection{Some vanishing statements}
\label{sub:sec:vanishing_statements}
To prove that $d$ is well-defined we will need the following proposition:
\begin{proposition}
\label{prop:differential_strictly_regular_elements}
    Let $S$ be a smooth proper surface. Let $b\in \EP{m+2}{S}$ be strictly regular. Then for any alteration $\varphi \colon \wt Y\to Y$ and any divisor $E\subset \wt Y$ contracted under $\varphi$ we have $[E, \ts_E(\ph^*(a))]=0\in \L(\k, m)_{j+1}$. 
\end{proposition}

\begin{remark}
    It is easy to see that the condition that $b$ is strictly regular is essential. Indeed, let $S\to \mb P^2$ be a blow-up of some point and let $E\subset S$ be the exceptional divisor. It is easy to see that the tame-symbol map $\ts_E\colon\Lambda^{m+2}\k(S)^\times\to\Lambda^{m+1}\k(E)^\times$ is surjective. Therefore, if the previous lemma were satisfied without the condition that $b$ is strictly regular, then for any $a\in\L^{m+1}\k(\P^1)^\t$, the element  $[\P^1,a]\in \L(\k,m)_{m-1}$ would vanish. However, by Theorem \ref{th:Lambda_is_generated_by_rationally_connected} such elements generate the group $\L(\k,m)_{m-1}$. So the group $\L(\k,m)_{m-1}$ would vanish. And Theorem \ref{th:Totaro_isomorphism_intro} would imply that the differential in the polylogarithmic complex $\Gamma_{\geq m-1}(\k,m)$ would be zero. However, it is easy to see that the differential in this complex is nonzero.
\end{remark}

The previous proposition follows from the following two lemmas:

\begin{lemma}
\label{lemma:finitnes_of_sum}
Let $S$ be a smooth proper surface. Let $b\in \Lambda^{m+2} L^\t$ be strictly regular. For any alteration $\ph\colon \wt S\to S$ and any divisor $E\subset \wt S$ contracted under $\ph$, the element $\ts_E(\ph^*(b))$ belongs to the subgroup generated by the elements of the form $f\wedge c_2\wdw c_{m+1}$, where $f\in \k(E)^\times$ and $c_i\in \k^\times$.
\end{lemma}

The proof of this lemma is essentially contained in the proof of Lemma 2.16 from \cite{bolbachan_2023_chow}. We remark that the proof of Lemma 2.16  from \cite{bolbachan_2023_chow} used Lemma 2.14 which proof was not correct. However, Lemma 2.14 is a particular case of Lemma \ref{lemma:characterisation_of_strictly_regular_elements} from this paper, which proof was given  in Section \ref{sec:preliminary_def}. The same remark applies to results stated in Section 4.3 and Section 7 of this paper.

\begin{proof}
    Denote by $z\subset E$ the image of $E$ under $\varphi$. We apply Lemma \ref{lemma:characterisation_of_strictly_regular_elements}. We can assume that $a$ has the form $f_1\wedge f_2\wedge g_3\wdw g_{m+2}$ such that $g_i$ are regular at $z$ and take non-zero values at this point. This implies that the functions $\res{\varphi^*(g_i)}{E}$ are non-zero constants. Now, the statement follows from Lemma \ref{lemma:leibniz_rule_tame_symbol} for $k=2$.
\end{proof}

\begin{lemma}
\label{lemma:Beilinson_Soule_vanishing_weight_zero}
Let $X$ be a smooth projective curve, $f\in\k(X)$ and $c_2,\dots, c_{m+1}\in\k^\t$. We have:
        $$[X, f_1\wedge c_2\wdw c_{m+1}]=0\in \L(\k,m)_{m-1}.$$
\end{lemma}

\begin{proof} 
Consider the multiplication map $\mu\colon \Lambda(\k,0)_0\otimes\Lambda^{m}\k^\times\to \L(\k,m)_{m-1}$ given by the formula
$$[X,f]\times (a_1\wdw a_{m})\mapsto[X,f\wedge a_1\wdw a_{m}].$$

The statement of this lemma says that $\mu$ is zero. To prove this, it is enough to show that $\Lambda(\k,0)_0=0$. So we can assume that $m=0$. We need to check that for any $f\in\k(X)^\times$ we have
$$[X,f]=0\in \Lambda(\k,0)_0.$$

Consider the following two cases:

\begin{enumerate}
    \item $f$ is constant. Choose some morphism $\psi\colon X\to\P^1$ of finite degree. We get
    $$[X,c]=(\deg \psi)[\P^1, c].$$
    So we can assume that $X=\P^1$. Choose some morphism $\psi_2\colon \P^1\to\P^1$ of degree $2$. We get:
    $$[\P^1,c]=(\deg \psi_2)(\P^1,c)=2[\P^1,c].$$
    So $[\P^1,c]=0$.

    \item Let $f$ be non-constant. We get:
    $$[X,f]=(\deg f)[\P^1,t].$$
    In this formula $t$ is the standard coordinate on $\P^1$. Let $\psi\colon \P^1\to\P^1$ be a map given by the formula $\psi(z)=z^{-1}$. We get:
    $$[\P^1,t]=(\deg \psi)^{-1}[\P^1,\psi^*(t)]=[\P^1, t^{-1}]=-[\P^1,t].$$
    So $[\P^1,t]=0$ and hence $[X,f]=0$.
\end{enumerate}
\end{proof}

The following lemma will be used in Section \ref{subsec:lambda_higher_chow_groups}.

\begin{lemma}
\label{lemma:differential_exceptional_divisor_zero_Bloch}
    Let $S$ be a proper surface and $f_1,\dots, f_n$ be regular maps from $S$ to $\mathbb P^1$ considered as rational functions on $S$. Let $D_i^1=f_i^{-1}(1)$. Set $Z=\cup D_i^1$ and $U=S\bs Z$. Denote by $D_i^0, D_i^\infty$ the divisors of zeros and poles of the functions $f_i$ on $U$. Assume that the divisors $D_1^0, D_1^\infty,\dots, D_n^0, D_n^\infty$ do not have common components and any three such divisors do not have common points. Let $\wt S$ be a smooth variety and $\ph\colon \wt S\to S$ be an alteration. Then for any divisor $E\subset \wt S$ contracted under $\ph$ we have $$[E, \ts_E(\ph^*(f_1\wdw f_n))]=0\in \L(\k, n-2)_{n-3}.$$

\end{lemma}

\begin{proof}
We can assume that $n\geq 1$. Denote by $w$ the image of $E$ under $\ph$. Assume that $w\subset Z_i$ for some $i$. Then the rational function $\res{\ph^*(f_i)}{E}$ is equal to $1$ and the statement is obvious. Assume that $w$ is not contained in $Z$. Without loss of generality, we can assume that $w\notin D_i^0,D_i^1$ for any $i\geq 3$. This implies that there are elements $c_3,\dots, c_n\in \k^\times$ so that $\res{\varphi^*(f_i)}{E}=c_i$. Now the statement follows from Lemma \ref{lemma:leibniz_rule_tame_symbol} and the previous lemma.
\end{proof}

The following lemma will be used in Section \ref{sec:W_surjective}.

\begin{lemma}
\label{lemma:about_degenerate_cycles}

Let $X$ be a proper variety and $f_1,\dots, f_n\in\k(X)^\times$. Denote by $s$ the transcendence degree of the extension $\k(f_1,\dots, f_n)\subset \k(X)$. Assume that $s>0$. Then
$$[Y, f_1,\dots, f_n]=0\in \Lambda(\k,n-\dim Y)_{n-2\dim Y}.$$
\end{lemma}

\begin{proof}Let $L=\k(f_1,\dots, f_n)$. There is an algebraic variety $X'$, a rational map $\psi\colon X\ard X'$ and rational functions $\wt f_1,\dots, \wt f_n\in\k(X')^\times$ such that $f_i=\psi^*(\wt f_i)$. Taking resolution of singularities, we can assume that $\psi$ is regular.

Choose a transcendence basis $t_1,\dots, t_s$ of $\k(X)$ over $\k(X')$. The extension
    $\k(X')(t_1,\dots, t_s)\subset \k(X)$ is finite. This implies that there is a rational map $g\colon X\ard X'\t(\P^1)^s$ of finite degree such that $\psi=\pi\circ g$, where $\pi\colon X'\t(\P^1)^s\to X'$ is the natural projection. Taking resolution of singularities, we can assume that $g$ is regular. We get:
    \begin{align*}
        [X,f_1\wdw f_n]=(\deg g)[X'\times (\P^1)^s, \wt f_1\wdw \wt f_n].       
    \end{align*}

    Thus we have reduced the statement to the case when $X=X\t(\P^1)^s, s>0$ and $f_i$ are functions on $X'$. Let $\theta$ be some morphism $\theta\colon (\P^1)^s\to (\P^1)^s$ of degree $2$. We get
    \begin{align*}
        &[X'\t (\P^1)^s, \wt f_1\wdw \wt f_n]=\\
        &1/2[X'\t (\P^1)^s,  (id\t\theta)^*(\wt f_1\wdw \wt f_n)]=\\
        &1/2[X'\t (\P^1)^s, \wt f_1\wdw \wt f_n].
    \end{align*}

    It follows that $[X'\t (\P^1)^s, \wt f_1\wdw \wt f_n]=0$.

\end{proof}

\subsection{The differential is well-defined}
\label{sub:sec:Lambda+correctness}

The goal of this section is to prove the following theorem:

\begin{theorem}
\label{th:Lambda(m)_well_defined}
    Let $X$ be a variety. The complex $\L(\k, m)$ is well defined.
\end{theorem}

 Let $Y_1$ and $Y_2$ be smooth varieties and $\ph\colon Y_1\to Y_2$ be an alteration. Let $D_1\subset Y_1, D_2\subset Y_2$ be prime divisors such that $\varphi(D_1)=D_2$. We recall that $e(D_1,D_2)$ denotes the multiplicity of $\varphi^*(D_2)$ along $D_1$ and $f(D_1,D_2)$ denotes the degree of the extension $\k(D_2)\subset \k(D_1)$. We need the following lemma.

\begin{lemma}
\label{lemma:functoriality_of_residue}
    Let $Y_1$ and $Y_2$ be smooth varieties and $\ph\colon Y_1\to Y_2$ be an alteration. Let $D_1\subset Y_1, D_2\subset Y_2$ be prime divisors such that $\varphi(D_1)=D_2$. For any $a\in\Lambda^n\k(Y_2)^\times$ we have
    $$[D_1,\ts_{D_1}(\ph^*(a))]=e(D_1, D_2)f(D_1, D_2)[D_2,\ts_{D_2}(a)].$$
\end{lemma}

\begin{proof}
    Denote the natural map $D_1\to D_2$ by $\psi$. We have:
    $$[D_2,\ts_{D_2}(a)]=(\deg \psi)^{-1}[D_1, \psi^*(\ts_{D_2}(a))].$$
As $f(D_1,D_2)=\deg\psi$, we get

$$e(D_1, D_2)f(D_1, D_2)[D_2,\ts_{D_2}(a)]=e(D_1, D_2)[D_1,\psi^*(\ts_{D_2}(a))].$$

So it remains to prove that

$$e(D_1, D_2)\psi^*(\ts_{D_2}(a))=\ts_{D_1}(\ph^*(a)).$$

This is a statement of Lemma \ref{lemma:functoriality_of_residue_2}.
\end{proof}

\begin{proposition}
\label{prop:d_is_well_defined}
    Let $j\in\{m-2,m-1\}.$ The map $d\colon \L(\k,m)_j\to \L(\k, m)_{j+1}$ is well-defined. 
\end{proposition}

\begin{proof}
When $j=m-1$, the statement is obvious. So we can assume that $j=m-2$. Let $a\in\EP{m+2}{S}$. By Lemma \ref{lemma:strictly_regular_blow_up} there is a proper birational morphism $\ph\colon \wt S\to S$ such that  the element $\ph^*(a)$ is strictly regular.
We recall that the element $d([S, a])$ is defined by the formula
$$\sum\limits_{D\subset \wt S}[D,\ts_{D}(\ph^*(a))].$$
We need to prove that this element does not depend on $\ph$ and is functorial with respect to alterations between surfaces. 

Let $\ph_1\colon \wt S_1\to S$ and $\ph_2\cl\wt S_2\to S$ be two such morphisms. It follows from Lemma \ref{lemma:strictly_regular_blow_up} that there are proper birational morphisms $\ph_{31}\colon \wt S_3\to S_1, \ph_{32}\colon\wt S_{3}\to S_2$ such that $\ph_1\circ\ph_{31}=\ph_{2}\circ\ph_{32}$ and the element $(\ph_{1}\circ\ph_{31})^*(a)=(\ph_{2}\circ\ph_{32})^*(a)$ is strictly regular. This shows that we only need to check the following statement. Let $\ph\colon \wt S\to S$ be an alteration and $a\in \EP{n}{Y}$. Assume that the elements $a$ and $\ph^*(a)$ are strictly regular. Then the following formula holds:
$$\sum\limits_{D\subset S}[D, 
\ts_D(a)]=(\deg \varphi)^{-1}\sum\limits_{D'\subset \wt S}[D',\ts_{D'}(\ph^*(a))].$$

We have:
\begin{align*}
          &\sum\limits_{D'\subset \wt S}[D',\ts_{D'}(\ph^*(a))]=\\
          &=\sum\limits_{D'\in Div(\ph)_{cont}}[D',\ts_{D'}(\ph^*(a))]+\sum\limits_{D\subset S}\sum\limits_{D'\in Div(\ph, D)}[D',\ts_{D'}(\ph^*(a))].
\end{align*}

In this formula $Div(\ph)_{cont}$ is the set of divisors contracted under $\ph$ and $Div(\varphi, D)$ the set of divisors whose image under $\varphi$ is equal to $D$. Each term in the first sum vanishes due to Proposition  \ref{prop:differential_strictly_regular_elements}. By Lemma \ref{lemma:functoriality_of_residue} and Lemma \ref{lemma:degree_is_a_sum_multiplicties}   the second term is equal to
        \begin{align*}
            &\sum\limits_{D\subset S}\sum\limits_{D'\in Div(\ph, D)}e(D', D)f(D',D)[D,\ts_{D}(a)]=\\&\sum\limits_{D\subset S}[D,\ts_{D}(a)]\sum\limits_{D'\in Div(\ph, D)}e(D', D)f(D',D)=\\
            &\deg(\ph)\sum\limits_{D\subset S}[D,\ts_{D}(a)].
        \end{align*}
        The proposition is proved.
\end{proof}

To finish the proof of Theorem \ref{th:Lambda(m)_well_defined} we need to show that $d^2=0$. This is a direct corollary of the following theorem:

\begin{theorem}
\label{th:Parshin_reciprocity_law}
    Let $S$ be a smooth complete surface. Let $a\in\EP{m+2}{S}$ be strictly regular. We have
    $$\sum\limits_{D\subset S}\sum\limits_{z\in D}\ts_{z}\ts_{D}(a)=0\in \Lambda^m\k^\times.$$
    
    Here the first sum is taken over all prime divisors on $S$. It is clear that in this sum there are only finitely many non-zero terms.
\end{theorem}

The point of this theorem is that for given $z$ there are precisely two $D$ such that $z\in D$ and $\ts_{z}\ts_{D}(a)\ne 0$. Moreover, these two elements have opposite signs. The particular case of this theorem when $m=2$ was proved in \cite[Theorem 2.15]{bolbachan_2023_chow}. The proof of the general case is similar. The case $m=0$ follows from classical Parshin reciprocity law \cite{parshin1975class}.

\begin{proof}[The proof of Theorem \ref{th:Lambda(m)_well_defined}]
    By Proposition \ref{prop:d_is_well_defined} we already know that $d$ is well-defined. Let us prove that $d^2=0$. We need to show that $d^2([S,a])=0$. By Lemma \ref{lemma:strictly_regular_blow_up}, there is a proper birational map $\varphi\colon \wt S\to S$ so that $\varphi^*(a)$ is strictly regular. So $[S,a]=[\wt S,\varphi^*(a)]$. This shows that we can assume that $a$ is strictly regular. 
    We have 
    $$d^2([S,a]) = \sum\limits_{D\subset S}d([D,\ts_D(a)]).$$
    Now the statement follows from the fact that
        $$d([D,\ts_D(a)])=\sum_{z\in D}[\spec\k,\ts_{z}\ts_{D}(a)].$$
and Theorem \ref{th:Parshin_reciprocity_law}.
\end{proof}

\subsection{The complex $\L(\k,m)$ and the Bloch's higher Chow groups}
\label{subsec:lambda_higher_chow_groups}
We recall that we have given the definition of cochain complex calculated  cubical higher Chow group $\mc Z_\square^m(\k,2m-*)$ and its alternating version $\mc N(\k,m)$ in Section \ref{subsec:higher_chow_groups}.

We recall that $\sigma_{\geq m-2}$ denotes the stupid truncation. Define a morphism of complexes $$\mc W'\colon \sigma_{\geq m-2}\mc Z_\square^m(\k,2m-*)\to\L(\k,m)_*$$ as follows. Let $[Z]\in  \mc Z_\square^m(\k,2m-j)$ be an irreducible cycle. Denote by $\overline Z$ the closure of this cycle in $(\P^1)^n$. Define
$$\mc W'(Z)=[\overline Z,  \wt t_1\wedge\dots\wedge \wt t_n].$$
In this formula $\wt t_i$ are the restrictions of the standard coordinates on $(\P^1)^n$ to $\overline Z$. If $(x_i:y_i)$ are homogenious coordinates on the $i$-th $\P^1$ then $t_i=x_i/y_i$. Finally define $\mc W$ as the composition $\sigma_{\geq m-2}\mc N(\k,m)_*\to \sigma_{m-2}\mc Z_\square^m(\k,2m-*)\xrightarrow{\mc W'}\L(\k,m)_*$. Here is the main result of this subsection:

\begin{theorem}
\label{th:W(m)_is_morphism}
    The map $\mc W'$ is a morphism of complexes. In particular, the map $\mc W$ is a morphism of complexes as well.
\end{theorem}

Let $Y$ be a variety. For a Cartier divisor $D$ on $Y$, denote the corresponding Weil divisor by $[D]$. 
We need the following lemma:

\begin{lemma}
\label{lemma:degree_of_morphism}
    Let $Y_1$, $Y_2$ be varieties and $\varphi\colon Y_1\to Y_2$ be an alteration. Then for any Cartier divisor $D$ on $Y_2$ we have
    $$\varphi_*[\varphi^*(D)]=(\deg\varphi)[D].$$ 
    In this formula $\ph^*$ is pullback on Cartier divisors and $\ph_*$ is pushforward on Weil divisors.
\end{lemma}

\begin{proof}
    As the statement is local on $Y_2$, we can assume that $D$ is principal. In this case the statement follows from \cite[Proposition 1.4]{fulton2013intersection}.
\end{proof}

\begin{proof}[The proof of Theorem \ref{th:W(m)_is_morphism}]
    Denote by $\wt t_i\in\k(\ol Z)$ the restriction of the coordinate function $t_i$ to $\ol Z$. Let $a=\wt t_1\wdw \wt t_n$ and $a_i = \wt t_1\wdw \wt t_{i-1}\wedge \wt t_{i+1}\wdw \wt t_n$. Let  $\ph\colon \wt Z\to \overline Z$ be a proper birational morphism with smooth $\wt Z$ such that the element $\wt a:=\ph^*(a)$ is strictly regular. We have
    $$\mc W'(Z)=[\overline Z,a].$$
    \begin{align*}
        &d(\mc W'(Z))=\sum\limits_{D'\subset \wt Z}[D',\ts_{D'}(\wt a)]=\sum\limits_{D'\in \Div(\ph)_{cont}}[D',\ts_{D'}(\wt a)]+\sum\limits_{D'\in \Div(\ph)_{gen}}[D',\ts_{D'}(\wt a)].
    \end{align*}

    In this formula $Div(\ph)_{cont}$ is the set of divisors contracted under $\ph$ and $Div(\ph)_{gen}$ is the set of  divisors which are not contracted under $\ph$. By Lemma \ref{lemma:differential_exceptional_divisor_zero_Bloch}, the first sum is equal to zero. For any divisor $D\subset \overline Z$, the set of divisors $D'\subset \wt Z$ such that $\ph(D')=D$ is denoted by $Div(\ph, D)$. We get
    $$\sum\limits_{D'\in Div(\ph)_{gen}}[D',\ts_{D'}(\wt a)]=\sum\limits_{D\subset \overline Z}\sum\limits_{D'\in \Div(\ph, D)}[D',\ts_{D'}(\wt a)].$$

We recall that $d([Z])$ is defined as
$$\sum_{i=1}^n(-1)^{i+1}(\ts_{i,0}([Z])-\ts_{i,\infty}([Z])),$$
where $\ts_{i,\varepsilon}$ is the cycle associated to the closed subset $Z\cap\{x_i=\varepsilon\}$, considered as a cycle on $(\mb P^1\bs\{1\})^{n-1}$. By the definition of a cycle associated to a closed subset, we get

\begin{align*}
    &\ts_{i,0}([Z])-\ts_{i,\infty}([Z])=\sum_{D\subset Z}\ord_D(\wt t_i)[\wt t_1,\dots, \wt t_{i-1},\wt t_{i+1},\dots, \wt t_n]_D=\\
    &\sum_{D\subset \overline Z}\ord_D(\wt t_i)[\wt t_1,\dots, \wt t_{i-1},\wt t_{i+1},\dots, \wt t_n]_D.
\end{align*}
The last formula holds because if $D\subset \ol Z\bs Z$ then one of $\wt t_i$ is identically equal to $1$ and so the expression $\ord_D(\wt t_i)[\wt t_1,\dots, \wt t_{i-1},\wt t_{i+1},\dots, \wt t_n]_D$ vanishes. We get:
    $$\mc W'(d(Z))=\sum\limits_{i=1}^n\sum\limits_{\substack{D\subset \overline Z\\\ord_D(\wt t_i)\ne 0}}(-1)^{i+1}\ord_{D}(\wt t_i)[D,i_D^*(a_i)].$$
    In this formula $i_D$ is the embedding of $D$ into $\overline Z$. As $Z$ intersects all the faces properly, for any divisor $D\subset \overline Z$ there is at most one $i$, such that $\ord_{D}(\wt t_i)\ne 0$. Fix $D$ and consider the following two cases:
    \begin{enumerate}
        \item For any $i$ we have $\ord_{D}(\wt t_i)=0$. It follows from Lemma \ref{lemma:degree_of_morphism} that for any $D'\in Div(\varphi, D)$ we have $\ord_{D'}(\varphi^*(\wt t_i))=0$. So $\ts_{D'}(\wt a)=0$ and we get $[D',\ts_{D'}(\wt a)]=0$.
        \item There is some $i$ such that $\ord_{D}(\wt t_i)\ne 0$. It remains to check the following formula:
            $$(-1)^{i+1}\ord_{D}(\wt t_i)[D,i_D^*(a_i)]=\sum\limits_{D'\in Div(\ph, D)}[D',\ts_{D'}(\wt a)].$$
    Denote by $g_{D'}$ the natural map $D'\to \ol Z$.
    As $\ord_{D'}(\ph^*(\wt t_j))=0$ for any $j\ne i$, it follows from the definition of tame symbol(see Proposition \ref{prop:tame_symbol_classiacal}) that
    \begin{align*}
        [D',\ts_{D'}(\wt a)]=(-1)^{i+1}\ord_{D'}(\ph^*(\wt t_i))[D', g_{D'}^*(a_i)]=\\
        (-1)^{i+1}\ord_{D'}(\ph^*(\wt t_i))f(D',D)[D,i_D^*(a_i)].
    \end{align*}
    It remains to show that
    $$\ord_{D}(\wt t_i)=\sum\limits_{D'\in Div(\ph, D)}f(D',D)\ord_{D'}(\ph^*(\wt t_i)).$$
    This follows from Lemma \ref{lemma:degree_of_morphism}.
    \end{enumerate}
\end{proof}

\subsection{The generators of $\L(\k,m)$}
\begin{theorem}[Proposition \ref{pr_intro:Lambda_is_generated_by_rationally_connected}]
\label{th:Lambda_is_generated_by_rationally_connected}
    Let $\k$ be an algebraically closed field of characteristic zero. The group $\L(\k, m)_{m-1}/\im(d)$ is generated by the elements of the form $[\P^1, a],$  $a\in \EP{m+1}{\P^1}$.
\end{theorem}

Essentially, this statement follows from \cite[Theorem 2.6, (i)]{bolbachan_2023_chow}; however, we will give a proof for completeness.

\begin{proof}

Denote by $A$ the subgroup of $\L(\k,m)_{m-1}$ generated by the elements of the form $[\P^1, w], w\in \L^{m+1}\k(\P^1)^\times$. Let $[X,a]\in \L(\k,m)_{m-1}$. We need to show that $[X,a]$ belongs to the subgroup generated by $A$ and $\im(d)$. 

Let $F=\k(t)$. Choose a polynomial $P$ such that $F[x]/P\cong \k(X)$. We identify $\k(X)$ with $F[x]/P$. In particular, we can consider $a$ as an element in the group $\Lambda^{m+1}(F[x]/P)^\times$.

For an irreducible polynomial $Q\in F[x]$ denote by $\nu(Q)$ the corresponding discrete valuation of the field $F(x)$. Consider the following map
    $$\Lambda^{m+2}F(x)\xrightarrow{(\ts_{\nu(Q)})}\bigoplus\limits_{Q}\Lambda^{m+1}(F[x]/Q)^\times.$$
    The sum is taken over all irreducible polynomials $Q\in F[x]$ and the map is given by tame-symbol maps $\ts_{\nu(Q)}$. By Lemma 2.8 from \cite{bolbachan_2023_chow} (see also \cite{MILNOR1969/70}) this map is surjective. This implies that there is $b\in \Lambda^{m+2}F(x)$ such that $\ts_{\nu(P)}(b)=a$ and for any $Q\ne P$ we have $\ts_{\nu(Q)}(b)=0$.

Let $S=\P^1\times\P^1$. We identify $F(x)$ with $\k(S)$. The discrete valuation $\nu(P)$ corresponds to some curve $X'\subset S$ birationally isomorphic to $X$. By construction, $\ts_{X'}(b)=a$ and for any non-rational prime divisor $D\subset S$ different from $X'$ we have $\ts_{D}(b)=0$.

Let $\varphi\colon \wt S\to S$ be a proper birational morphism such that the element $\varphi^{*}(b)$ is strictly regular. Let us compute $d([\wt S, \varphi^*(b)])$. We get
$$\sum\limits_{D\subset \wt S}[D,\ts_D(\varphi^*(b))]=\sum\limits_{D\in \Div(\varphi)_{cont}}[D,\ts_D(\varphi^*(b))]+[X,a]+\sum\limits_{\substack{D\subset S\\D\ne X'}}[D,\ts_D(b)].$$

In this formula, we denoted by $\Div(\varphi)_{cont}$ the set of prime divisors contracted under $\varphi$. The first sum lies in $A$ because each prime divisor contracted under $\varphi$ is rational. The second sum belongs to $A$ by construction. It follows that $[X,a]$ belongs to the subgroup generated by $A$ and $\im(d)$. 
\end{proof}

\section{The map $\wW$ is an isomorphism}
\label{sec:W_isomorphsim}
Define a subcomplex $\M(\k,m)\subset \N_{\geq m-1}(\k,m)$ as follows. The vector space $\M(\k,m)_{m-1}$ is generated by the elements of the form $[f_1,f_2,\dots, f_n]_X$, where: $X$ is a complete curve, the divisors of the functions $f_1, f_2$ are disjoint and the functions $f_i,i\geq 3$ are constants. The vector space $\M(\k,m)_m$ is defined as $d(\M(\k,m)_{m-1})$. In the case $m=2$ this complex was defined in \cite{cubical1999herbert}. The subcomplex $\M(\k,m)$ is defined for any field. The field $\k$ is not necessary algebraically closed. In the next subsection we will show that $\M(\k,m)$ is acyclic. 

Assume now that $\k$ is algebraically closed. We need the following lemma.
\begin{lemma}
\label{lemma:Beilinson_Soule_vanishing}
Let $X$ be a smooth projective curve, $f_1,f_2\in\k(X)$ and $c_3,\dots, c_{m+1}\in\k^\t$. We have:
        $$[X, f_1\wedge f_2\wedge c_3\wdw c_{m+1}]=0\in \L(\k,m)_{m-1}.$$
\end{lemma}

\begin{proof} We can assume that $m=1$. If $f_1$ is constant then the statement follows from Lemma \ref{lemma:Beilinson_Soule_vanishing_weight_zero}. So we can assume that $f_1$ is not constant.

The extension $\k(f_1)\subset\k(X)$ is finite. As any separable extension is contained in some finite Galois extension, there is a finite extension $\k(X)\subset L$ so that the corresponding extension $\k(f_1)\subset L$ is Galois. Choose a smooth projective curve $Y$ together with a morphism $\psi\colon Y\to X$ such that $\k(Y)\cong L$ and the field extension $\k(X)\subset L$ is induced by $\psi$. We get
$$[X,f_1\wedge f_2]=(\deg\psi)^{-1}[Y,\psi^*(f_1)\wedge \psi^*(f_2)].$$

So we can assume that $f_1$ is non-constant and the field extension $\k(f_1)\subset \k(X)$ is Galois. Denote the Galois group by $G$. We have
$$[X,f_1\wedge f_2]=1/|G|\sum\limits_{g\in G}[X,g^*(f_1\wedge f_2)]=1/|G|\sum\limits_{g\in G}[X,f_1\wedge g^*(f_2)]=1/|G|[X,f_1\wedge \wt f_2],$$
where $\wt f_2=\prod\limits_{g\in G}g^*(f_2)$. As $\wt f_2$ is invariant under the Galois group, it follows that there is $h\in\k(t)$ so that $\wt f_2=h(f_1)$. We get

$$1/|G|[X,f_1\wedge \wt f_2]=1/|G|[X,f_1\wedge h(f_1)]=(\deg f_1)/|G|[\P^1,t\wedge h(t)]=[\P^1,t\wedge h(t)].$$

In the last formula we used that $|G|=\deg f_1$.

As $h(t)$ is product of linear factors, it is enough to consider the following two cases:

        \begin{enumerate}

            \item $X=\P^1, f_1=t, f_2=t-a, a\in\k$. Let $\ph$ be an automorphism of $\mb P^1$ given by the formula $\ph(t)=a-t$. We get
            \begin{align*}
                &[\P^1, t\wedge (a-t)]=[\P^1, \ph^*(t\wedge (t-a))]=\\&[\P^1, (a-t)\wedge (-t)]=[\P^1, (t-a)\wedge t]=-[\P^1, t\wedge (t-a)].
            \end{align*}
            So $[\P^1, t\wedge (a-t)]=0$.
            \item $X=\P^1, f_1=t, f_2=c, c\in\k$. In this case the statement follows from Lemma \ref{lemma:Beilinson_Soule_vanishing_weight_zero}.
  \end{enumerate} 

\end{proof}

We recall that we assume that $\k$ is algebraically closed. We recall that $\wL(\k,m)=\tau_{\geq m-1}\L(\k,m)$. Let $\wt \N(\k,m)=(\tau_{\geq m-1}\N(\k,m))/\M(\k,m)$. By the previous lemma, the map $\W$ induces a well-defined map $\wCH(\k,m)\to\wL(\k,m)$ which we denote by $\wW$. Here is the main result of this section:
\begin{theorem}[Theorem \ref{th:W_isomorphism_intro}]
    Let $\k$ be an algebraically closed field. The map $\wW$ is an isomorphism.
\end{theorem}

\subsection{An acyclic subcomplex}
In this subsection we do not assume $\k$ to be algebraically closed. The goal of this subsection is to prove that $\M(\k,m)$ is acyclic.

\begin{lemma}
\label{lemma:homotopy_left_inverse}
   Assume that $g_1,g_2,c_2,\dots, c_m\in\k^\t$. The following cycle
$$[g_1g_2,c_2,\dots, c_m]_{\spec\k}-[g_1,c_2,\dots, c_m]_{\spec\k}-[g_2,c_2,\dots, c_m]_{\spec\k}$$
is contained in $\M(\k,m)_m$. 
\end{lemma}
\begin{proof}
    Consider the following cycle:
$$Z=\left[\dfrac{(t-g_1g_2)(t-1)}{(t-g_1)(t-g_2)},t,c_2,\dots, c_m\right]_{\P^1}.$$
We claim that the boundary of this cycle is equal to the element stated in the lemma. Indeed, the cycle $Z$ does not intersect faces given by the equations $z_i=0,z_i=\infty$ when $i>1$. The intersection of this cycle with the face given by $z_1=0$ is given by
$$[g_1g_2,c_2,\dots, c_m]_{\spec\k}.$$
The intersection of $Z$ with the face given by $z_1=\infty$ is given by
$$[g_1,c_2,\dots, c_m]_{\spec\k}+[g_2,c_2,\dots, c_m]_{\spec\k}$$
the statement follows from the definition of the differential in the complex $\CH(\k,m)$.
\end{proof}

\begin{proposition}
\label{lemma:M_is_acyclic}
    The complex $\M(\k,m)$ is acyclic.
\end{proposition}

\begin{proof}
Let us first explain why higher Chow groups satisfy Galois descent. Let $\ph\cl K\emb F$ be a (not necessarily finite) Galois extension with Galois group $G$. We need to prove the following statement
    $$H^i(\N(\k,m))\cong H^i(\N(F,m))^G.$$
It is enough to prove that 

    $$\N(\k,m)\cong \N(F,m)^G.$$
We have 
    $$\N(F,m)=\colim_{F_0} \N(F_0,m),$$
    where the colimit is taken over the set of subextension $\k\subset F_0\subset F$ such that $F_0/\k$ is finite and Galois. This reduces the statement to the particular case when $F/K$ is finite. Denote by $\psi\colon \spec F\to \spec \k$ the natural map. As $\psi$ is proper and flat we have canonical maps $\psi_*\colon N(F,m)\to \N(\k,m)$ and $\psi^*\colon \N(\k,m)\to\N(F,m)$ (see \cite{bloch1986algebraic}). The Galois descent follows from the following two formulas:
    $$\psi^*\psi_*=\sum\limits_{g\in G}g_*$$
    $$\psi_*\psi^*=[F:K].$$

The above discussion shows that we can assume that $\k$ is algebraically closed. Let $$A=\bigoplus\limits_{m=0}^\infty \mc N(\k,m).$$
There is a natural structure of graded commutative dg-algebra on $A$ \cite{bloch_rriz_1994_mixed}, where $\mc N(\k,m)_j$ sits in degree $j$. Let $Z_1\subset \square^{2m_1-j_1}$ and $Z_2\subset \square^{2m_2-j_2}$. The multiplication is given by the formula $[Z_1]\cdot [Z_2]=alt_{2(m_1+m_2)-j_1-j_2}([Z_1\times Z_2])$. 

If $C_*$ is a chain complex, denote by $S^m C_*$ its $m$-th symmetric power. It is defined as the quotient of $C_*^{\otimes m}$ by the relations of the form $$a_1\otimes\dots \otimes a_{k+1}\otimes a_{k}\otimes\dots \otimes a_m-(-1)^{|a_k||a_{k+1}|}a_1\otimes\dots\otimes a_{k}\otimes a_{k+1}\otimes\dots \otimes a_m=0.$$
In the last formula we denoted by $|a_i|$ the degree of the element $a_i$. 

Denote by $V$ the chain complex $S^m\N(\k,1)$. There is a natural map $S^m\mc N(\k,1)\to \mc N(\k,m)$ given by multiplication. By \cite[Theorem 6.1]{bloch1986algebraic} we have $H^j(\N(\k,1))=0$ for $j\ne 1$ and $H^1(\N(\k,1))=\k^\times$. It follows from the Kunneth formula that $H^j(V)=0$ for $j\ne m$ and $H^m(V)=\Lambda^m \k^\times$. Define the complex $V'$ as follows. If $i< m$ we set $V'_i=V_i$. If $i=m$ we define $V'_m$ as the kernel of the natural map $V_m\to \Lambda^m\k^\times$. It follows that the complex $V'$ is acyclic. 

    Denote by $\Q[\k^\times]$ a free vector space generated by all elements of $\k^\times$.  We have
    $$V'_{m-1}\cong \mc N(\k,1)_0\otimes \Lambda^{m-1}\Q[\k^\times],\quad V'_{m}\cong \ker(\Lambda^m\Q[\k^\times]\to \Lambda^m\k^\times).$$
    Consider the natural map $\tau_{\geq m-1}V'\to \N_{\geq m-1}(\k,m)$ given by multiplication. Let us show that the image of this map is contained in $\M(\k,m)$. In degree $m-1$ this is straightforward. In degree $m$ this follows from the previous lemma(or from the fact that any element in $V'_m$ is exact).
    
    Denote by $\varphi$ the natural map $V'_{\geq m-1}\to \M(\k,m)$. It is easy to see that $\varphi$ is surjective in degree $m-1$ and injective in degree $m$. As $H^m(\M(\k,m))=0$ it follows that $\varphi$ is isomorphism. As $V'$ is acyclic it follows that $\M(\k,m)$ is acyclic as well.
\end{proof}

We set $\wCH(\k,m)=\tau_{\geq m-1}(\CH(\k,m))/\M(\k,m)$.

\begin{lemma}
\label{lemma:Galois_descent}
    The cohomology $H^i(\wCH(\k,m))$ satisfies Galois descent.
\end{lemma}

\begin{proof}

Let $F/\k$ be an arbitrary Galois extension. We have the natural map $\ph^*\cl \tau_{\geq m-1}\CH(K,m)\to \tau_{\geq m-1}\CH(F,m)$. Moreover $\ph^*(\M(K,m))\subset \M(F,m)$. So we get the following diagram:
    \begin{equation}
        \begin{tikzcd}
             H^i(\tau_{\geq m-1}\CH(K,m)) \ar[r]\ar[d] & H^i(\tau_{\geq m-1}\CH(F,m))^G\ar[d]\\
             H^i(\wCH(K,m)) \ar[r]& H^i(\wCH(F,m))^G
        \end{tikzcd}
        \end{equation}
    Vertical arrows are isomorphisms by Proposition \ref{lemma:M_is_acyclic}. Top arrow is isomorphism because higher Chow groups rationally satisfies Galois descent(see the beginning of the previous proof). This implies that the bottom arrow is an isomorphism as well.
\end{proof}

%\begin{remark}
%    The complex $\wCH(\k,m)$ is functorial under arbitrary extensions of fields and satisfies Galois descent. One can show that the same is true for the complex $\L(\k,m)$. This implies that the map $\wW$ is a quasi-isomorphism for any field.
%\end{remark}

\subsection{Multiplicative relations}
Starting from here till the end of Section \ref{sec:strong_suslin_reciprocity_law} we assume that $\k$ is algebraically closed.
\begin{proposition}
\label{prop:multiplicative_relations_Chow}
    Let $X$ be a smooth projective curve and $g_1, g_2, f_3\dots, f_{m+2}\in \k(X)$. Set $g_3=g_1g_2$. Assume that the cycles $W_i=[g_i,f_3,\dots, f_{m+2}]_X$ are admissible. Then $W_3=W_1+W_2$ in $\wCH(\k,m)_{m-1}$.
\end{proposition}

In the case $m=2$ this proposition was stated in \cite[Lemma 2.2]{cubical1999herbert}.

\begin{lemma}
    Let $(x,t)$ be coordinates on $X\t\P^1$. If the conditions of the previous theorem hold, then the following cycle
        $$\left[\dfrac{(t-g_3(x))(t-1)}{(t-g_1(x))(t-g_2(x))},t, f_3(x),\dots, f_{m+2}(x)\right]_{X\t\P^1}$$
    is admissible.
\end{lemma}

\begin{proof}
We set $f_1=(t-g_3)(t-1)/((t-g_1)(t-g_2))$ and $f_2=t$. Let $S=X\t\P^1$. Denote by $\psi$ a rational map $S\to\mb (\P^1)^{n+1}$ given by the formula $z\mapsto (f_1(x,t),\dots, f_{n+1}(x,t))$. Denote by $Z$ the closure of its image. If $\dim Z<2$, there is nothing to prove. We can assume that $\dim Z=2$. Let $\ph\colon \wt S\to S$ be a resolution of the indeterminacy locus of $\psi$. In particular, the inverse image of the indeterminacy locus of $\psi$ is a snc divisor on $\wt S$. Set $\psi'=\psi\circ\varphi$. This is a regular map $\wt S\to (\P^1)^{n+1}$. We can assume that for any prime divisor $D\subset \wt S$ contracted under $\ph$, the point $\ph(D)$ lies on the indeterminate locus of $f_1$. Denote by $\pi_i\colon Z\to \P^1$ the natural projections and set $f_i'=f_i\circ \ph$. Let $D\subset S$ be a prime divisor. Denote by $D'$ the strict transform of $D$ under $\ph$. If $D'$ is contracted under $\psi'$, we set $\psi(D)=0$. Otherwise, we set $\psi(D)=\psi'(D')$ and call $\psi(D)$ \emph{the image} of $D$ under $\psi$.

Let us first check that $Z$ intersects faces of codimension $2$ properly.

Let $D\subset \wt S$ be a prime divisor.  We will say that $D$ is good if at least one of the following statements hold:
\begin{enumerate}
    \item For some $i$, the restriction of $f_i'$ to $D$ is equal to $1$.
   \item For any $i$ we have $\ord_{D}(f_i')=0$
    \item There is precisely one $i$ such that $\ord_D(f_i')\ne 0$ and for any $x\in D$ at least one of the following conditions holds:
    \begin{enumerate}
       \item there is some $j\ne i$ such that  $f_j'(x)=1$ (we recall that for any $j$, the function $f_j'$ is regular on $\wt S$ as a map to $\P^1$).
        \item There is $j_0\ne i$, such that for any $j\not\in\{i,j_0\}$ we have $f_j'(x)\not\in\{0,\infty\}$.
    \end{enumerate}
\end{enumerate}

To prove the statement of the lemma it is enough to prove that any divisor $D\subset\wt S$, not contracted under $\psi'$ is good. Indeed, assume that there is some curve $D\subset Z\cap\square^{n+1}$ lying on two faces. Denote by $D'$ the closure of $D$ in $(\P^1)^{n+1}$. Then the proper transform $\wt D$ of $D'$ under $\psi'$ would satisfy $\ord_{\wt D}(f_i'),\ord_{\wt D}(f_j')\ne 0$ for some $i\ne j$. So $\wt D$ wouldn't be good. A contradiction. Similarly, using the third part of the definition of a good divisor, one can show that a point lying in $Z\cap\square^{n+1}$ cannot lie on three different faces. Starting from here we assume that any divisor $D\subset \wt S$ is not contracted under $\psi'$.

 Let $D$ be some irreducible component of the divisor $(f_1')$. The divisor $D$ is contracted under the map $f_1'$. So $D$ is not contracted under $f_i'$ for some $i>1$. This implies that $D'$ is not contracted under $\ph$. 

Let us prove that any divisor $D\subset \wt S$ not contracted under $\psi'$ is good. First, consider the case when $D$ is contracted under $\ph$. Let $(x_0,t_0)=\ph(D)$.  Assume that there is $i$ such that $\ord_D(f_i')\ne 0$. By what we have said before we know that $i>1$. Consider two cases:
\begin{description}
\item[Case $i=2$] We know that $t_0\in\{0,\infty\}$. As $(x_0,t_0)$ lies on the indeterminacy locus of $f_1$ it is easy to see that $x_0$ lies on the divisor of $g_j$ for some $j$. As the cycles $W_i$ are admissible, one of the following cases occurs:
\begin{enumerate}
    \item For some $j\geq 3$ we have $f_j(x_0)=1$. In this case the restriction of $f_j'$ to $D$ is equal to $1$ and $D$ is good.
    \item For any $j\geq 3$ we have $f_j(x_0)\in \k^\t\bs\{1\}$. In this case, the restrictions of $f_j', j\geq 3$ to $D$ are non-zero constants. As we have said before, $\ord_{D}(f_1')=0$. This implies that $D$ is good.
\end{enumerate}
    \item[Case $i\geq 3$] Without loss of generality we can assume that $i=3$. We know that $x_0$ lies on the divisor of $f_3$. As the cycles $W_i$ are admissible, one of the following cases occurs:
    \begin{enumerate}
    \item For any $j$ we have $g_j(x_0)=1$. As $f_1$ is undefined at $(x_0, t_0)$, we have $t_0=1$. So the restriction of $f_2'$ to $D$ is equal to $1$ and $D$ is good.
        \item For some $j>3$ we have $f_j(x_0)=1$. In this case the restriction of $f_j'$ to $D$ is equal to $1$ and $D$ is good.
        \item For any $j>3$ we have $f_j(x_0)\in \k^\t\bs\{1\}$. This implies that $g_j(x_0)\in\k^\t$. As $(x_0,t_0)$ lies on the indeterminacy locus of $f_1$ we know that $t_0\in\k^\t$. So $\ord_D(f_1')=0$ and the functions $\res{f_2'}{D}, \res{f_j'}{D}, j\geq 4$ are non-zero constants. This implies that $D$ is good.
    \end{enumerate}
\end{description}

So we can assume that $D$ is not contracted under $\ph$. Let $D'=\ph(D)$. Denote by $T_i$ the closure of the set given by the equation $f_i=1$ on $S$. Let $T=\cup T_i, T'=\cup_{i>1}T_i$ and $U=S\bs T, U'=S\bs T'$. If $D'$ is contained in $T$ then the restriction of one of $f_i'$ to $D$ is equal to $1$ and so $D$ is good. So, we can assume that $D'\cap U\ne\varnothing$. If for any $i$ we have $\ord_D(f_i')=0$ then $D$ is good. So, we can assume that there is some $i$ such that $\ord_D(f_i')\ne 0$.

 Let us prove that for any $j\ne i$, we have $\ord_{D}(f_j')=0$. This follows from the fact that the divisors of rational functions $f_i$ on $U$ have no common components. 

Let $z\in D$ and $z'=\ph(z)$. It remains to check that one of the conditions (a) or (b) from the definition of a good divisor holds. If $z\not\in U'$, then the restriction of $f_j'$ for some $j>1$ to $z$ is equal to $1$. So we can assume that $z'\in U'$. Consider several cases:

\begin{enumerate}
\item Let $z'=(x_0,t_0)$ lie on the indeterminacy locus of $f_1$. 
\begin{description}
\item[Case $i=1$] 
\begin{enumerate}
    \item $D'=\ol{\{t=g_k\}}$ for some $k$. We can choose $x$ as a coordinate on $D$. The restriction of $f_2'$ to $D$ is equal to $g_k$ and the restriction of $f_j',j\geq 3$ to $D$ is equal to $f_j$. So the statement follows from the fact that the cycle $[g_k,f_3,\dots, f_{n+1}]_X$ is admissible.
    \item $D'$ is an irreducible component of $(g_k)$ for some $k$. Choose $t$ as coordinate on $D$. The restrictions of $f_j', j\geq 3$ to $D$ are non-zero constants. So for any $j\geq 3$ we have $f_j'(z)\in\k^\t, j\geq 3$.
\end{enumerate}
    \item[Case $i=2$] We have $t_0\in \{0,\infty\}$. As $z'$ lies on the indeterminacy locus of $f_1$, there is $k$,  such that $g_k(x_0)\in\{0,\infty\}$. As $z'\in U'$ and the cycles $W_i$ are admissible this implies that for any $j\geq 3$ we have $f_j(x_0)\in \k^\t$. So for $j\geq 3$ we have $f_j'(z)\in\k^\t.$
    \item[Case $i\geq 3$] We can assume that $i=3$. We know that $x_0$ lies on the divisor of $f_3$. As the cycles $W_i$ are admissible, one of the following cases occurs:

    \begin{enumerate}
        \item For any $j$ we have $g_j(x_0)=1$. As $f_1$ is undefined in $z'$, we have $t_0=1$. So $z'\not\in U'$.
\item For any $j$ we have $g_j(x_0)\in\k^\t$ and for any $l\geq 4$ we have $f_l(x_0)\in\k^\t$. As $z'$ lies on the indeterminacy locus of $f_1$ we have $t_0\in\k^\t$. So the restrictions of the functions $f_2', f_j', j\geq 4$ to $D$ are non-zero constants.       
    \end{enumerate}

\end{description}

    \item Suppose that $z'$ does not lie on the indeterminacy locus of $f_1$. If $z'\in T$ then $f_j(z')=1$ for some $j$ and so $f_j'(z)=1$. So we can assume that $z\in U$. It is enough to check that $z'$ does not lie on divisors $(f_{i_1}), (f_{i_2}), (f_{i,3})$ for some $i_1<i_2<i_3$. The case $i_2\geq 3$ is obvious, because, in this case, the divisors $(f_{i_1}), (f_{i_2})$ are disjoint on $U$. So we can assume that $i_1=1, i_2=2, i_3=3$. Assume that $z'$ lies on the divisors of $f_2$ and $f_3$. As $x_0$ lies on the divisor of $f_3$, we have $g_j(x_0)\in \k^\t$. As $t_0$ lies on the divisor of $f_2$, we have $t_0\in\{0,\infty\}$. This implies that $z'$ does not lie on the divisor of $f_1$.
\end{enumerate}

\end{proof}

\begin{proof}[The proof of Proposition \ref{prop:multiplicative_relations_Chow}]
We can assume that $m \geq 2$.  Consider the cycle
 $$Z=\left[\dfrac{(t-g_3)(t-1)}{(t-g_1)(t-g_2)},t, f_3,\dots, f_{m+2}\right]_{X\t\P^1}.$$
By the previous lemma $Z$ is admissible. Let us compute its differential. Let $\varphi\colon \wt S\to S$ be a resolution of indeterminacy locus of $f_1$. Set $f_i'=f_i\circ\varphi$. Let $D'$ be some irreducible component of the intersection of $Z$ with one of the faces. Let $D$ be the proper transform of $D'$ under $\varphi$. $D$ is a divisor on $\wt S$.

Assume first that $D$ is contracted under $\ph$. As $D$ is contracted under $f_i',i\geq 2$ the restriction of $f_i'$ to $D$ are constants. This implies that the corresponding term in the differential lies in $\mc M(\k,m)$.

So we can assume that $D$ is not contracted under $\varphi$. If $t=0,\infty$ we get zero. If $x$ lies on the divisor of $f_i, i\geq 3$ then $f_j(x), j\ne i, j\geq 3$ are constants and the corresponding term in the differential lies in $\M(\k,m)$. If $t=1$, we get $0$. If $t=g_k$ we get $[g_k, f_3,\dots, \dots, f_{m+2}]_{\P^1}$. If $x$ lies on the divisor of $g_k$ for some $k$, then all $f_j, j\geq 3$ are constants and the corresponding term lies in $\M(\k,m)$.
\end{proof}

\subsection{Basic moving lemma}
\label{sec:moving_lemma}

\begin{definition}
    Let $X$ be a smooth variety. We say that a divisor $D=\sum\limits_{i=1}^n n_i D_i$ is general with respect to a family $\mc C=\{C_1\dots, C_n\}$ of smooth closed subsets, if the following conditions hold:
    \begin{enumerate}
        %\item $D_i$ does not intersect singular locus of $C_j$ for any $i,j$.
        \item $D$ is supported on a snc divisor.
        \item For any $j_1<\dots< j_k$ and any $i$, the smooth variety $D_{j_1}\cap\dots\cap D_{j_k}$ intersects each $C_i$ transversally.
    \end{enumerate}
\end{definition}

We have the following classical statement \cite[Theorem 8.18]{hartshorne2013algebraic}.

\begin{theorem}[Bertini Theorem]
    Let $X$ be a smooth projective variety and $D$ be a very ample divisor. Denote by $|D|$ the set of effective divisors which are equivalent to $D$. There is an open subset $U\subset |D|$ such that for any $H\in U$, the divisor $H$ is smooth. Moreover, if $\dim X>1$, then $H$ is also irreducible.
\end{theorem}

\begin{lemma}
\label{lemma:moving_lemma_Bertini}
    Let $X$ be a smooth projective variety. Let $D$ be a divisor and $\mc C$ be any family of smooth closed subsets. There is a divisor $D'$ rationally equivalent to $D$ such that $D'$ is general with respect to $\mc C$.
\end{lemma}

 \begin{proof} We will prove only the case $\dim X>1$. The case $\dim X=1$ is similar.
 
 Let us first assume that $D$ is very ample. In this case, the statement follows from Bertini theorem simultaneously applied to $X$ and each $C_i$. Moreover, we can assume that $D'$ is irreducible.
 
As $X$ is smooth projective, any $D$ can be represented as $D_1-D_2$ with $D_1$ and $D_2$ be very ample. As $D_1$ is very ample we can apply the statement of this lemma to $D_1$. Let $D_1'$ be such that $D_1\sim D_1'$ and $D_1'$ is general with respect to $\mc C$. Denote by $\wt{\mc C}$ the union of $\mc C\cup \{D_1'\}\cup \{X\cap D_1'|X\in \mc C\}$. As $D_2$ is very ample we can apply the statement of this lemma to $D_2$  and $\wt{\mc C}$. Let $D_2'$ be a divisor such that 1)$D_2\sim D_2'$ and 2)$D_2'$ is general with respect to $\wt{\mc C}$. We argue that the divisor $D'=D_1'-D_2'$ satisfies the condition of the lemma.
 \end{proof}

The goal of this subsection is to prove the following lemma.

\begin{lemma}[Moving lemma]
\label{lemma:moving_lemma_final}
    Let $X$ be a smooth projective variety. Let $D$ be a smooth prime divisor and $\mc C$ be a family of smooth closed subsets. Assume that $D$ is general with respect to $\mc C$. Then there is a rational function $f$ such that:
    \begin{enumerate}
    \item $\ord_D(f)=1$;
        \item $(f)$ is general with respect to $\mc C$.
    \end{enumerate}
\end{lemma}

\begin{proof}
    Apply the previous lemma to $D$ and $\wt{\mc C}=\mc C\cup \{D\}\cup \{X\cap D|X\in \mc C\}$. Let $D=D'+(f)$. As $D'$ is general with respect to $\wt C$ we have $\ord_D(f)=1$ and the divisor $(f)$ is general with respect to $\mc C$.
\end{proof}

\begin{lemma}[Moving lemma for curves, strong form]
\label{lemma:moving_lemma_strong}
    Let $X$ be a smooth projective curve. Let $x\in X$ and $D$ be a finite subset of $X\bs\{x\}$. Then there is a rational function $f\in\k(X)$ such that 1) $\ord_x(f)=1$ and 2) for any $y\in D$ we have $f(y)=1$.
\end{lemma}

\begin{proof}
Let $x_1=x$ and $D=\{x_2,\dots, x_n\}$.
By previous lemma there is a function $f_0$ such that $\ord_{x_1}(f_0)=1$ and $f_0(x_i)\in\k^\t$ for $i>1$. Let $a_i=f_0(x_i)^{-1}$. It remains to show that there is a function $g$ such that $g(x_1)=1, g(x_i)=a_i, i>1$, because in this case we can set $f=f_0g$.

Let us show that for any $i$ there is a function $g_i$ such that $g_i(x_i)=1$ and $g_i(x_j)=0$ for any $j\ne i$. 

Indeed, by previous lemma there is a function $r_{i}$ such that $r_{i}(x_i)=0$ and $r_i(x_j)\in \k^\t$ for any $j\ne i$. Let $\wt g_i=r_1\cdot\dots r_{i-1}\cdot r_{i+1}\cdot\dots r_n$. We have $\wt g_i(x_i)\in \k^\t$ and $\wt g_i(x_j)=0$ for any $j\ne i$. Finally, set $g_i=\wt g_i/\wt g_i(x_i)$.

Now we can set $g=1+(a_2-1)g_2+\dots (a_k-1)g_k$. 
\end{proof}

\subsection{The map $\mc W$ is surjective}
\label{sec:W_surjective}
The goal of this subsection is to prove the following theorem:

\begin{theorem}
\label{th:W_is_surjective}
    Let $\k$ be algebraically closed and $j\in\{m-2,m-1,m\}$. The map $\mc W\cl\N(\k,m)_j\to\L(\k,m)_j$ is surjective.
\end{theorem}

We recall that the notion of an admissible cycle was introduced in Subsection \ref{subsec:higher_chow_groups}.

\begin{definition}
    We will say that a sequence of functions $(f_1,\dots, f_n)$ is weakly admissible if $\supp((f_1),\dots, (f_n))$ is a snc divisor.

    We will say that $(f_1,\dots, f_n)$ is strictly admissible if it is weakly admissible and $(f_i)$ do not have common components.

    We will say that a sequence of functions $f_1,\dots, f_n$ is admissible if the cycle $[f_1,\dots, f_n]_X$ is admissible. 
\end{definition}

\begin{remark}
    A sequence which is not weakly admissible can be admissible. However, a  strictly admissible sequence is admissible and weakly admissible.
\end{remark}

We need the following lemma whose proof is deferred:

\begin{lemma}
\label{lemma:globally_strictly_regular}
    Let $X$ be a smooth projective variety. Assume that the sequence $(f_1,\dots, f_n)$ is weakly admissible. Then $f_1\wdw f_n$ is a linear combination of the elements of the form $g_1\wdw g_n$ such that the sequence $(g_1,\dots, g_n)$ is strictly admissible.
\end{lemma}

\begin{proof}[The proof of Theorem \ref{th:W_is_surjective}] Let $a=[X,f_1\wdw f_n]\in\L(\k,m)_j$. We need to show that $a$ lies in the image of $\mc W$. By Hironaka resolution of singularities, we can assume that $X$ is smooth projective and the sequence $f_1,\dots, f_n$ is weakly admissible. Using the previous lemma we can assume that the sequence $(f_1,\dots, f_n)$ is strictly admissible. Consider the map $\psi\colon X\to (\mb P^1)^n$ given by $z\mapsto (f_1(z),\dots, f_n(z))$. This map is undefined at points where the divisor of zeros of some $f_i$ intersects the divisor of poles of $f_i$. Blowing up $X$ in these codimension $2$ subsets, we can assume that $\psi$ is regular.  Denote by $Z$ its image. 

Assume that $\dim Z<\dim X$. In this case it follows from Lemma \ref{lemma:about_degenerate_cycles}, that $[X,f_1\wdw f_n]=0$. In particular, $[X,f_1\wdw f_n]$ lies in the image of $\mc W$.

So we can assume that $\dim Z=\dim X$. It follows that the natural map $X\to Z$ is an alteration. As the sequence $f_1,\dots, f_n$ is strictly admissible, the cycle $[f_1,\dots, f_n]_X$ is admissible. Now, it is easy to see that $$[X,f_1\wdw f_n]=\mc W([f_1,\dots, f_n]_X).$$
\end{proof}

We proceed to the proof of Lemma \ref{lemma:globally_strictly_regular}. We need the following definition.
\begin{definition}
\label{def:J}
    Let $X$ be a smooth variety and $f_1,\dots, f_n\in\k(X)$. Let $D$ be a prime divisor. Define the set $J_D=J_D(f_1,\dots, f_n)$ as follows. Let $J_D'$ be $\{i|\ord_D(f_i)\ne 0\}$. If $|J_D'|\geq 2$ we set $J_D=J_{D}'$. Otherwise, we set $J_D=\varnothing$.
\end{definition}

Let us explain the idea of the proof of Lemma \ref{lemma:globally_strictly_regular}. Let $(f_1,\dots,f_n)$ be a weakly admissible sequence. Repeatedly applying identities of the form 
$$f_1\wedge\dots\wedge f_i\wdw f_n=f_1\wdw (f_i/h)\wdw f_n+f_1\wdw h\wdw f_n,$$
we will show the element $f_1\wdw f_n$ can be represented as a linear combination of elements of the form $g_1\wdw g_n$ such that the sequence $(g_1,\dots, g_n)$ is weakly admissible and (for fixed $(g_1,\dots, g_n)$) the sets $J_D(g_1,\dots, g_n)$ are disjoint, where $D$ varies among all prime divisors on $X$. So we can assume that $(f_1,\dots, f_n)$ is a weakly admissible sequence such that the sets $J_D(f_1,\dots, f_n)$ are disjoint. We have the following lemma.
\begin{lemma}
\label{lemma:formula}
Let $X$ be a variety  and $h, f_1,\dots, f_n\in \k(X)^\t$. Then:
\begin{align*}
    &f_1\wedge f_2\wdw f_n=\\&1/2(f_1/h\wedge f_2\wdw f_n+f_1\wedge f_2/h\wdw f_n+f_1/f_2\wedge h\wdw f_n).
\end{align*}
\end{lemma}

We will show that the repeated application of the above lemma allows to write element $f_1\wdw f_n$ as a linear combination of the elements of the form $g_1\wdw g_n$ such that the sequence $(g_1,\dots, g_n)$ is strictly admissible.

\begin{proof}[The proof of lemma \ref{lemma:formula}]
    Indeed, we have
    \begin{align*}
        &f_1/h\wedge f_2\wdw f_n+f_1\wedge f_2/h\wdw f_n+f_1/f_2\wedge h\wdw f_n=\\
        &f_1\wedge f_2\wdw f_n-h\wedge f_2\wdw f_n+f_1\wedge f_2\wdw f_n-\\
        &-f_1\wedge h\wedge f_3\wdw f_n+f_1\wedge h\wedge f_3\wdw f_n-f_2\wedge h\wedge f_3\wdw f_n=\\
        &2\left(f_1\wdw f_n\right)-h\wedge f_2\wdw f_n-f_2\wedge h\wdw f_n=\\
        &2\left(f_1\wdw f_n\right).
    \end{align*}
\end{proof}

\begin{proof}[Proof of Lemma \ref{lemma:globally_strictly_regular}.]
    Let $D$ be an effective divisor and $D=\sum\limits_{i=1}^nn_iD_i$. Denote by $s(D)$ the family of subsets $\{D_{i_1}\cap\dots\cap D_{i_k}| 1\leq i_1<\dots< i_k\leq n\}$.
    
    Let $(f_1,\dots,f_n)$ be a weakly admissible sequence. Let us first show that $f_1\wdw f_n$ can be written as a linear combination of elements of the form $g_1\wdw g_n$ such that 1)the sequence $(g_1,\dots, g_n)$ is weakly admissible and 2)the sets $J_D$ are disjoint, where $D$ varies among all prime divisors on $X$. Define the following invariant:    
$$a_1(f_1,\dots, f_n)=\sum\limits_{D_1\ne D_2}|J_{D_1}\cap J_{D_2}|.$$ 

The vanishing of $a_1$ is equivalent to the fact that the sets $J_D$ are disjoint. So, it is enough to prove that the element $f_1\wdw f_n$ can be represented as a linear combination of elements of the form $g_1\wdw g_n$ such that 1)the sequence $(g_1,\dots, g_n)$ is weakly admissible and 2) $a_1(g_1,\dots, g_n)<a_1(f_1, \dots, f_n)$.

Let us suppose that $J_{D_1}\cap J_{D_2}\ne\varnothing.$ Without loss of generality we can assume that $1\in J_{D_1}\cap J_{D_2}$. Let $D=\supp((f_1),\dots, (f_n))$ and $D'=D-[D_1]$. By Lemma \ref{lemma:moving_lemma_final} there is $h$ such that $\ord_{D_1}(h)=1$ and $(h)$ is general with respect to $s(D')$. We have
    $$f_1\wedge f_2\wdw f_n=(f_1h^{-\ord_{D_1}(f_1)})\wedge f_2\wdw f_n+(\ord_{D_1}(f_1))(h\wedge f_2\wdw f_n).$$
    We argue that $((f_1h^{-\ord_{D_1}(f_{1})}),f_2,\dots, f_n), (h,f_2,\dots, f_n)$ are weakly admissible and have strictly smaller value of $a_1$.

    So we can assume that the sets $J_D$ are disjoint. Define the following invariant:
    $$a_2(f_1,\dots, f_n)=\sum\limits_{D}\sum\limits_{i<j}|\ord_D(f_i)\ord_D(f_j)|.$$ 
    
    The vanishing of $a_2$ is equivalent to the fact that $(f_1,\dots, f_n)$ is strictly admissible. So, it is enough to show that $f_1\wdw f_n$ is a linear combination of elements of the form $g_1\wdw g_n$ such that 1)the sequence $(g_1,\dots, g_n)$ is weakly admissible, 2)the sets $J_D(g_1,\dots, g_n)$ are disjoints and 3)$a_2(g_1,\dots, g_n)<a_2(f_1,\dots, f_n)$. 
    
    Without loss of generality, we can assume that for some $D_1$ we have $\ord_{D_1}(f_1), \ord_{D_1}(f_2)>0$. Let $D=\supp((f_1),\dots, (f_n))$ and $D'=D-[D_1]$. Choose $h$ such that $\ord_{D_1}(h)=1$ and $h$ is general with respect to $s(D')$. Such $h$ exists by Lemma \ref{lemma:moving_lemma_final}.  By Lemma \ref{lemma:formula}, we get
    $$f_1\wedge f_2\wdw f_n=1/2(f_1/h\wedge f_2\wdw f_n+f_1\wedge f_2/h\wdw f_n+f_1/f_2\wedge h\wdw f_n).$$
    We argue that the sequences $(f_1/h,\dots, f_n), (f_1,f_2/h\dots, f_n), (f_1/f_2,h,\dots, f_n)$ satisfy the conditions stated above.
\end{proof}

\subsection{Higher Chow groups and strictly admissible sequences}

Denote by $\CH(\k,m)_{m-1}'$ a subgroup of $\CH(\k,m)_{m-1}$ generated by the elements of the form $[f_1,\dots, f_n]_X$, where $X$ is a smooth projective curve, and $(f_1,\dots, f_n)$ is strictly admissible sequence(= the divisors of $f_i$ are disjoint). Denote by $\pi$ the natural projection $\CH(\k,m)_{m-1}\to\wCH(\k,m)_{m-1}.$ We recall that $\mc W$ denotes the natural map $\sigma_{\geq m-2}\N(\k,m)\to \L(\k,m)$ constructed in Section \ref{subsec:lambda_higher_chow_groups}.

\begin{theorem}
\label{th:Chow_genereted_disjoint_divisors}
    The group $\CH(\k,m)_{m-1}$ is generated by the subgroups $\CH(\k,m)_{m-1}'$ and $\ker\pi\cap\ker\W$.
\end{theorem}

\begin{remark}
    It follows from the main result of this section that $\ker \W\subset \ker \pi$. However we don't know this fact yet.
\end{remark}

\begin{remark}
    It follows from this theorem that the group $\CH(\k,m)_{m-1}/\ker\pi\cap\ker\W$ is generated by the elements of the form $[X,f_1,\dots, f_n]$, where $X$ is a smooth projective curve and the divisors of $f_i$ are disjoint. However, this theorem does not say that any cycle in higher Chow groups is a linear combination of smooth curves; The cycle $[X,f_1,\dots, f_n]$ can be singular. The fact that any $1$-dimensional cycle on $\square^n$ has the form $[Y,g_1,\dots, g_n]$ with smooth projective curve $Y$ follows from the fact that any curve has a resolution of singularities.
\end{remark}

The proof of this theorem is similar to the proof of Lemma \ref{lemma:globally_strictly_regular}. From here to the end of this subsection we work modulo elements from $\ker\pi\cap\ker\W$.

\begin{lemma}
    The group $\CH(\k,m)_{m-1}$ is generated by $\ker\pi\cap\ker\W$ and by elements of the form $[f_1,\dots, f_n]_X$ such that the sets $J_x(f_1,\dots, f_n)$ are disjoint, as $x$ varies among the closed points of $X$.
\end{lemma}

\begin{proof}
As any curve has a resolution of singularity, any irreducible cycle in $\mc N(\k,m)_{m-1}$ has the form $[f_1,\dots, f_n]_X$ for some smooth projective curve $X$. By definition, the sequence $(f_1,\dots, f_n)$ is admissible. Let
    $$a_1(f_1,\dots, f_n)=\sum\limits_{x\ne y} |J_{x}\cap J_y|.$$ We will show that the element $[f_1,\dots, f_n]_X$ is a linear combination of elements of the form $[g_1,\dots, g_n]_X$ such that the sequence $(g_1,\dots, g_n)$ is admissible and $$a_1(g_1,\dots, g_n)<a_1(f_1,\dots, f_n).$$ Suppose that $J_x\cap J_y\ne\varnothing$. We can assume that $1\in J_x\cap J_y$.
    Let $D=\supp((f_1),\dots, (f_n))$ and $D'=D-[x]$. Choose $h$ such that $h(y)=1$ for any $y\in \supp D'$ and $\ord_x(h)=1$. Such function exists by Lemma \ref{lemma:moving_lemma_strong}. As before, it is easy to see that the sequences $(f_1h^{-\ord_x(f_1)}, f_2,\dots, f_n), (h^{\ord_x(f_1)},f_2,\dots, f_n)$ are admissible and their values of $a_1$ are strictly smaller. It remains to show that the element $$[f_1,\dots,f_n]-[f_1h^{-\ord_x(f_1)}, \dots,f_n]+[h^{\ord_x(f_1)},\dots,f_n]$$
    is contained in $\ker\pi\cap\ker\mc W$.
    Indeed, it is obvious that this element is contained in $\ker\mc W$. The fact that this element is contained in $\ker \pi$ follows from the multiplicative relations proven by Lemma \ref{prop:multiplicative_relations_Chow}. 
\end{proof}

\begin{proof}[The proof of Theorem \ref{th:Chow_genereted_disjoint_divisors}]
Let $f_1,\dots, f_n$ be an admissible sequence such that the sets $J_x(f_1,\dots, f_n)$ are disjoint. Define the following invariant:
$$a_2(f_1,\dots, f_n)=\sum\limits_x\sum\limits_{i<j} |\ord_x(f_i)||\ord_x(f_j)|.$$
In this formula $x$ varies among all the closed points of $X$.
Assume that the value of $a_2$ is positive. We will show that the element $[f_1,\dots, f_n]_X$ is a linear combination of elements of the form $[g_1,\dots, g_n]_X$ such that 1) the sequence $(g_1,\dots, g_n)$ is admissible 2)the sets $J_x(g_1,\dots, g_n)$ are disjoint and 3) $$a_2(g_1,\dots, g_n)<a_2(f_1,\dots, f_n).$$ Suppose that $|\ord_{x}(f_1)||\ord_{x}(f_2)|>0$. Let $D=\supp((f_1),\dots, (f_n))$ and $D'=D-[x]$. Choose $h$ such that $\ord_{x}=1$ and $h(y)=1$ for any $y\in \supp D'$. It is easy to see that the following sequences are admissible: 
\begin{align*}
    &(h,f_2,\dots, f_n), (f_1/h,f_2,\dots, f_n), (h,f_1,f_3,\dots, f_n), \\
    & (h/f_2,f_1,\dots, f_n), (h,f_2/f_1,\dots, f_n).
\end{align*} 
We get:
\begin{align*}
    &[f_1,f_2,\dots, f_n]_X=[h,f_2,\dots, f_n]_X+[f_1/h,f_2,\dots, f_n]_X=\\
    &[h,f_1,\dots, f_n]_X+[h,f_2/f_1,\dots, f_n]_X+[f_1/h,f_2,,\dots, f_n]_X=\\
    &[f_2,f_1,\dots, f_n]_X+[h/f_2,f_1,\dots, f_n]_X+[h,f_2/f_1,\dots, f_n]_X+[f_1/h,f_2,,\dots, f_n]_X.
\end{align*}
So
\begin{align*}
    &[f_1,f_2,\dots, f_n]_X=\\
    &1/2([h/f_2,f_1,\dots, f_n]_X+[h,f_2/f_1,\dots, f_n]_X+[f_1/h,f_2,,\dots, f_n]_X).
\end{align*}
We remark that this formula essentially follows from Lemma \ref{lemma:formula}. It remains to note, as before, that the numbers $$a_2(h/f_2,f_1,\dots, f_n), a_2(h,f_2/f_1,\dots, f_n), a_2(f_1/h,f_2,,\dots, f_n)$$ are strictly smaller than $a_2(f_1,\dots, f_n).$
\end{proof}

\subsection{Generators and relations for $\EP{n}{X}$}
Let $X$ be a smooth projective curve. Denote by $V_n(X)$ a vector space generated by elements of the form $\langle f_1,\dots,f_n\rangle$, where $f_i$ are rational functions on $X$ with pairwise disjoint divisors, modulo the following relations:
\begin{enumerate}
    \item $\langle f_{\sigma(1)},\dots, f_{\sigma(n)}\rangle=sgn(\sigma)\langle f_1,\dots,f_n\rangle.$
    \item Assume that $(f_i)$ is disjoint from $(g_j)$ for any $i,j$ and $(f_i)$ is disjoint from $(f_j)$ for any $i\ne j$. Then $$\langle g_1g_2,f_2,\dots, f_n\rangle=\langle g_1,f_2,\dots, f_n\rangle+\langle g_2,f_2,\dots, f_n\rangle.$$ 
\end{enumerate}

We have the natural map $\theta_{\Lambda,n}\colon V_n(X)\to \EP{n}{X}$, given by the formula $\theta(\langle f_1,\dots, f_n\rangle)=f_1\wdw f_n$. 

Here is the main result of this subsection:

\begin{theorem}
\label{th:theta_is_an_isomorphsim}
    Let $X$ be a smooth projective curve over $\k$. Then the map $\theta_{\Lambda,n}$ is an isomorphism.
\end{theorem}

We know that this map is surjective by Lemma \ref{lemma:globally_strictly_regular}. So we need to check that $\theta_{\Lambda,n}$ is injective.

Let $Z$ be a finite subset of $X$. Denote by $T_{Z,n}(X)$ a vector space given by the following generators and relations. 

Generators: $\langle f_1,\dots,f_n\rangle_\otimes$, where $f_i$ are non-zero rational functions on $X$ with pairwise disjoint divisors such that for any $z\in Z$ we have $\ord_z(f_i)=0$ for any $i$. 

Relations: Let $k\in\{1,\dots, n\}$. Assume that $(f_i)$ is disjoint from $(g_j)$ for any $i,j$ and $(f_i)$ is disjoint from $(f_j)$ for any $i\ne j$. Then $$\langle f_1,\dots, f_k,g_1g_2,f_{k+2},\dots, f_n\rangle_\otimes=\langle f_1,\dots, f_k,g_1,f_{k+2},\dots, f_n\rangle_\otimes+\langle f_1,\dots, f_k,g_2,f_{k+2},\dots, f_n\rangle_\otimes.$$ 

In the case $Z=\varnothing$ we will write $T_{n}$ for $T_{Z,n}$.

%If $\sup((h))\subset Z$, we have the natural map $T_{h,Z\cup\sup((h)),n}\colon T_{Z,n-1}(X)\to V_{Z\bs \sup((h)),n}(X)$ given by the formula $\langle f_1,\dots,f_n\rangle_\otimes \mapsto \langle h,f_1,\dots,f_n\rangle_\otimes.$

Denote by $\theta_{\otimes,Z,n}$ the natural map $T_{Z,n}(X)\to \otimes^n\k(X)^\times$. When $Z=\varnothing$ we will write it as $\theta_{\otimes,n}$.

To prove Theorem \ref{th:theta_is_an_isomorphsim} we need the following lemma. 

\begin{lemma}
\label{prop:theta_tensor_is_injective}
    The map $\theta_{\otimes,Z,n}$ is injective.
\end{lemma}

\begin{proof}[The proof of Theorem \ref{th:theta_is_an_isomorphsim}]
    Assume that $\theta_{\Lambda,n}(\alpha)=0$. We have the natural maps $alt'\colon V_n(X)\to T_{n}(X)$ given by antisymmetrization:
    $$alt'(\langle f_1,\dots, f_n\rangle)=\dfrac 1{n!}\sum_{\sigma\in S_n}sgn(\sigma)\langle f_{\sigma(1)},\dots, f_{\sigma(n)} \rangle_\otimes.$$
    We can identify $\EP nX$ with the set of antisymmetric tensors in $\otimes^n\k(X)^\times$. Denote by  $$alt\colon \EP nX\to \otimes^n \k(X)^\times$$the corresponding map. We get $\theta_{\otimes, n}alt'(\alpha)=alt(\theta_{\Lambda,n}(\alpha))=0.$ So by the previous proposition we get $alt'(\alpha)=0$.

    Denote by $pr$ the natural map $T_{n}(X)\to V_n(X)$ given by the formula
    $$\langle f_1,\dots, f_n\rangle_\otimes\mapsto \langle f_1,\dots, f_n\rangle.$$
    We get $\alpha=pr(alt'(\alpha))=pr(0)=0$.
\end{proof}

To prove Lemma \ref{prop:theta_tensor_is_injective} we need the following two lemmas.

\begin{lemma}
\label{lemma:about_basis}
    Let $X$ be a smooth projective curve and $Z\subset X$ be a finite subset of $X$. There is a basis $h_\beta$ of the vector space $\k(X)^\times\otimes_\Z\Q$ with the following property.  Assume that for some $f\in \k(X)^\times$ we have the decomposition $$f=\prod_\beta  h_\beta^{n_\beta},\quad n_\beta\in \Q.$$
Assume also that for some $z\in Z$, we have $\ord_z(f)=0$. Then for any $\beta$ such that $n_\beta\ne 0$ we have $\ord_z(h_\beta)=0$.
\end{lemma}

\begin{proof}
    For any $z$ choose function $h_z$ such that $\ord_z(h_z)=1$ and $\ord_w(h_z)=0$ for any $w\in Z\bs\{z\}$. Such function exists by Lemma \ref{lemma:moving_lemma_strong}. Extend $h_z$ to a basis. Let $\{h_z,q_\beta\}$ be a basis that we get in this way. Set $$\wt q_\beta=q_\beta\prod_{z\in Z} h_z^{-\ord_z(q_\beta)}.$$
    For any $z\in Z$ we get $\ord_z(\wt q_\beta)=0$. 
It is easy to see that the basis $\{h_z,\wt q_\beta\}$ satisfies the statement of this lemma.
\end{proof}

The following lemma is well-known.

\begin{lemma}
\label{lemma:linear_independence}
    Let $V,W$ be vector spaces. Let $v_1,\dots, v_k\in V, w_1,\dots, w_k\in W$. Assume that $v_i$ are linearly independent. If
    $\sum v_k\otimes w_k=0$ then $w_k=0$.
\end{lemma}

\begin{proof}[The proof of proposition \ref{prop:theta_tensor_is_injective}]
    The proof is by induction on $n$. The case $n=1$ is obvious.

    Let $x\in T_{ Z,n}(X)$ such that $\theta_{\otimes,Z,n}(x)=0$. We need to check that $x=0$. Let
    $$x=\sum_{\alpha\in A} n_\alpha\langle f_1^\alpha\wedge\dots\wedge f_n^\alpha\rangle_\otimes.$$
    We can assume that $n_\alpha\ne 0$.

    Let $Q$ be the set $$\sup\left(\sum_{i,\alpha}\sup((f_i^\alpha))\right)$$ and let $\wt Z=Z\cup Q$. Let $h_\beta$ be a basis from Lemma \ref{lemma:about_basis} for $Z=\wt Z$.

    Replacing $x$ with $Nx, N\in\mb Z$ we can assume that coefficients of each $f_1^\alpha$ in this basis are integers. So for some $l_{\alpha,\beta}\in\Z$ we have
    $$f_1^\alpha=\prod {h_\beta}^{l_{\alpha,\beta}}.$$

    For fixed $\alpha$ denote by $B_\alpha$ the set of $\beta$ such that $l_{\alpha,\beta}\ne 0$. For fixed $\beta$ denote by $A_\beta$ the set of $\alpha$ such that $l_{\alpha,\beta}\ne 0$. Let $B=\cup_\alpha B_\alpha$. We get the following relations in the group $T_{Z,n}(X)$:

    $$\langle f_1^\alpha,\dots, f_n^\alpha\rangle_\otimes=\sum_{\beta\in B_\alpha} l_{\alpha,\beta}\langle h_\beta,f_2^\alpha,\dots, f_n^\alpha\rangle_\otimes.$$

    So we get:

    $$x=\sum_{\alpha\in A} n_\alpha\sum_{\beta\in B_\alpha} l_{\alpha,\beta}\langle h_\beta,f_2^\alpha,\dots,f_n^\alpha\rangle_\otimes=$$
    $$\sum_{\beta\in B} \sum_{\alpha\in A_\beta} n_\alpha l_{\alpha,\beta}\langle h_\beta,f_2^\alpha,\dots,f_n^\alpha\rangle_\otimes.$$

    Let $Z_\beta$ be the union of $Z$ and $\sup((h_\beta))$. Denote by $y_\beta$ the following element in the group $T_{ Z_\beta, n-1}$:
    $$y_\beta = \sum_{\alpha\in A_\beta} n_\alpha l_{\alpha,\beta}\langle f_2^\alpha,\dots,f_n^\alpha\rangle_\otimes.$$

We get:

    $$\theta_{\otimes, Z,n}(x)=\sum_{\beta\in B} h_\beta\otimes (\theta_{\otimes, Z_\beta, n-1}(y_\beta)).$$

    By previous lemma, we get $\theta_{\otimes, Z_\beta, n-1}(y_\beta)=0$. It follows from the inductive assumption that $y_\beta=0$.

Assume that $\beta\in B$. For any $z\in Z$ we have $\ord_z(h)=0$. Define the map $L_\beta\colon T_{Z_\beta, n-1}(X)\to T_{ Z, n}(X)$ given by the formula
$$\langle g_2,\dots, g_n\rangle_\otimes\mapsto \langle h_\beta, g_2,\dots, g_n\rangle_\otimes.$$

We get

$$x=\sum_{\beta\in B} L_\beta(y_\beta)=\sum_\beta L_\beta(0)=0.$$
\end{proof}

\subsection{The map $\wt{\mc W}$ is an isomorphism}
We recall that $\k$ is algebraically closed. It follows from the proof of Lemma \ref{lemma:M_is_acyclic} that $\wt N_{\geq m-1}(\k,m)_m$ is isomorphic to $\Lambda^m\k^\times$. The group $\L_{\geq m-1}(\k,m)_m$ is isomorphic to $\Lambda^m\k^\times$ by definition. Moreover, it is easy that the natural map $\wt{\mc W}\colon\wt N_{\geq m-1}(\k,m)_m\to \L_{\geq m-1}(\k,m)_m$ is an isomorphism. So $\wt{\mc W}$ is an isomorphism in degree $m$.
Let us show that it is also an isomorphism in degree $m-1$.
Let $X$ be a smooth projective curve. Define a map 
$$\ph_X\colon V_n(X)\to \wCH(\k,m)_{m-1}$$
by the formula
    $\ph_X(\langle f_1,\dots, f_n\rangle)=[f_1,\dots, f_n]_X$. The cycle is admissible; relations are satisfied by Proposition \ref{prop:multiplicative_relations_Chow}. 
     Define a map $\lambda\cl \L(\k,m)_{m-1}\to \wCH(\k,m)_{m-1}$
by the formula $\lambda([X, a])=\ph_{\wt X}(\theta_{\Lambda, n}^{-1}(\psi^*(a)))$, where $\psi\colon \wt X\to X$ is the normalization. It is easy to check that the relations defined $\L(\k,m)_{m-1}$ are satisfied. 

Denote by $\pi$ the natural projection $\CH(\k,m)_{m-1}\to \wCH(\k,m)$. Let us show that for any $a\in \CH(\k,m)_{m-1}$ we have $\lambda\W(a)=\pi(a)$. By Theorem \ref{th:Chow_genereted_disjoint_divisors} the group $\CH(\k,m)_{m-1}$ is generated by the subgroups $\CH(\k,m)_{m-1}'$ and $\ker\pi\cap\ker\W$. When $a\in \CH(\k,m)_{m-1}'$ the formula follows from the definition of the map $\theta_{\Lambda,n}$. When $a\in \ker\pi\cap\ker\W$ the formula is obvious as both sides are zero.

Let us show that $\lambda$ vanishes on the image of $d$. We need to show that for any $a\in \L(\k,m)_{m-2}$ we have $\lambda(d(a))=0$. By Theorem \ref{th:W_is_surjective} there is $a'\in \CH(\k,m)_{m-2}$ such that $a=\mc W(a')$. We get $\lambda(d(a))=\lambda(d(\mc W(a')))=\lambda(\mc W(d(a')))=\pi(d(a'))=0$.

So we get a map $\overline{\lambda}\colon \wL(\k,m)_{m-1}\to \wCH(\k,m)_{m-1}$. The map $\wW$ is surjective by Theorem \ref{th:W_is_surjective}. It is injective as $\lambda\circ\W=\pi$ and so $\overline{\lambda}\circ\wW=id$.

\section{Strong Suslin reciprocity law}
\label{sec:strong_suslin_reciprocity_law}
The goal of this section is to give some generalization of the main result of \cite{bolbachan_2023_chow}. As this generalization is rather straightforward, our arguments are only sketched and we refer the reader to \cite{bolbachan_2023_chow}.
We assume $\k$ to be algebraically closed. Let $F$ be some extension of $\k$. Denote by $\Gamma(F,m)$ the polylogarithmic complex of a field $F$. It is a cochain complex concentrated in degrees $[1,m]$ having the following form: 
$$\Gamma(F,m)\colon \mathcal B_m(F)\xrightarrow{\delta_m} \mathcal B_{m-1}(F)\otimes F^\times\xrightarrow{\delta_m}\dots\xrightarrow{\delta_m}\mathcal B_2(F)\otimes \Lambda^{m-2}F^\times\xrightarrow{\delta_m}\Lambda^m F^\times.$$
Let $(F,\nu)$ be a discrete valuation field. We recall that $\ol F_\nu$ denotes the residue field. The following proposition was proved in \cite{goncharov1995geometry}:

\begin{proposition}
\label{prop:tame_symbol}
Let $(F,\nu)$ be a discrete valuation field and $m\geq 2$. There is a unique morphism of complexes $\ts_\nu\colon \Gamma(F,m)\to \Gamma(\ol F_\nu, m-1)[-1]$ satisfying the following properties:

\begin{enumerate}
    %\item For any units $u_1,\dots u_n$ we have $\ts[n]_\nu(u_1\wedge\dots \wedge u_n)=0$.
   \item For any uniformiser $\pi$ and units $u_2,\dots, u_m\in F$ we have $\ts_\nu(\pi\wedge u_2\wedge\dots \wedge u_m)=\overline {u_2}\wedge \dots \wedge \overline{u_m}$.
    \item For any $a\in F\bs \{0,1\}$ with $\nu(a)\ne 0$, any integer $k$ satisfying $2\leq k\leq m$ and any $b\in \Lambda^{m-k}F^\t$ we have $\ts_\nu(\{a\}_k\otimes b)=0$.
    \item For any unit $u$, any integer $k$ satisfying $2\leq k\leq m$ and $b\in \Lambda^{m-k}F^\t$ we have $\ts_\nu(\{u\}_k\otimes b)=-\{\overline u\}_k\otimes \ts_\nu(b)$.
\end{enumerate}
\end{proposition}
The sign in the item $3$ comes from the fact that the element $\{u\}_k$ sits in degree $1$ which is odd (or corresponds to elements in $K_{2k+1}(F)$ and the number $2k+1$ is odd).

Denote by $\F_d$ the category of finitely generated (as fields) extensions of $\k$ of transcendence degree $d$. As any extension of finitely generated fields having the same degree is finite, any morphism in $\F_d$ is a finite extension.  For $F\in \F_d$, denote by $\dval(F)$ the set of discrete valuations given by a prime divisor on some birational model of $F$. When $F\in\F_1$ this set is equal to the set of all discrete valuations which are trivial on $\k$. In this case, we denote this set simply by $\val(F)$. Let $X$ be a smooth variety and $D\subset X$ be a prime divisor on $X$. Abusing notations, we denote by $\ts_D$ the map $\ts_{\nu_D}$, where $\nu_D$ is a discrete valuation corresponding to $D$. 

For $F\in\F_1$ denote
$$Tot_F:=\sum\limits_{\nu\in\val(F)}\ts_\nu\colon \tau_{\geq m}\Gamma(F, m+1)\to (\tau_{\geq m-1}\Gamma(\k, m))[-1].$$
It was proved in \cite{rudenko2021strong} that $Tot_F$ induces zero map on cohomology. In particular there is a homotopy between $Tot_F$ and the zero map.

\begin{definition}
\label{def:SRL}
    Let $F\in \F_1$. 
    \emph{A lifted reciprocity map} on the field $F$  is a map $$h\colon \L^{m+1}(F^\t)\to \Gamma(\k,m)_{m-1}/\im \delta_{m}$$
    such that:
    \begin{enumerate}
    \item The map $h$ gives a homotopy between $Tot_F$ and the zero map:
    \begin{equation}
    \label{diagram:SRL_def}
        \begin{tikzcd}
            (\B_2(F)\otimes \L^{m-1}F^\t)/         \im(\delta_{m+1})\ar[d,"Tot_F"]\ar[r,"\delta_{m+1}"]& \L^{m+1}F^\t \ar[d,"Tot_F"]\ar[dl,"h"]\\
            (\B_2(\k)\otimes \L^{m-2}\k^\t)/\im(\delta_m)\ar[r,"-\delta_{m}"] & \L^{m}\k^\t          
        \end{tikzcd}
    \end{equation}
        \item $h(f_1\wedge f_2\wedge c_3\wdw c_{m+1})=0$
        for any $f_i\in F$ and  $c_i\in\k$.

    \end{enumerate}
\end{definition}

Denote by $\Set$ the category of sets. 
Define a contravariant functor $$\SRL\colon \F_1\to \Set$$ as follows. For any $F\in \F_1$ the set $\SRL(F)$ is equal to the set of all lifted reciprocity maps on $F$. If $j\colon F_1\emb F_2$ then $\SRL(j)(h_{F_2})=h_{F_1}$,where $h_{F_1}(\alpha_1\wdw \alpha_{m+1}):=\dfrac 1{\deg j}h_{F_2}(j(\alpha_1)\wdw j(\alpha_{m+1}))$. The fact that in this way we get a functor can be proved similarly to \cite[Proposition 2.1]{bolbachan_2023_chow}.

 The goal of this section is to prove the following two theorems:

\begin{theorem}
\label{th:strong_suslin_reciprocity_law_field}
    To any field $F\in\F_1$, one can associate a lifted reciprocity map $\mc H_F$ on the field $F$ such that:
    \begin{enumerate}
        \item For any embedding $j\colon F_1\emb F_2$ in $\F_1$ we have $$\mc H_{F_1}(a)=(1/\deg j)\mc H_{F_2}(j(a)).$$
        \item We have: $$\mc H_{\k(t)}(t\wedge (1-t)\wedge (1-a/t)\wedge c_4\wdw c_{m+1})=-\{a\}_2\otimes c_4\wdw c_{m+1}.$$
    \end{enumerate}
    Moreover the family of the maps $\mc H_F$ are uniquely determined by the properties stated above.
    \end{theorem}

    \begin{theorem}
    \label{th:two_dimensional_reciprocity_law}
        \begin{enumerate}
            \item For any field $L\in\F_2$ and any $b\in\L^{m+2}(L^\t)$ we have
$$\sum\limits_{\nu\in\dval(L)}\mc H_{\ol L_\nu}(\ts_\nu(b))=0.$$
           \item Let $S$ be a smooth proper surface over $\k$ and $b\in\L^{m+2} k(S)^\t$, such that $b$ is strictly regular. Then

            $$\sum\limits_{D\subset S}\mc H_{\k(D)}(\ts_D(b))=0.$$
            The sum in this formula is taken over all prime divisors $D$ on $S$.
        \end{enumerate}
        
    \end{theorem}

\subsection{The outline of the proof}
In Section \ref{subsec:SRL_1} we will show that there is the unique lifted reciprocity map on the field of rational functions $\k(t)$. Denote it by $\mc H_{\k(t)}$. Let $F\in\F_1$ be any field. Choose some embedding $j\colon \k(t)\emb F$. To define $\mathcal H_F$, we extend a lifted reciprocity map $\mathcal H_{\k(t)}$ from the field $\k(t)$ to the field $F$. For this we solve the more general problem: for any finite extension $j'\colon F_1\emb F_2$ in $\F_1$ we construct the canonical map $N_{F_2/F_1}\colon \SRL(F_1)\to \SRL(F_2)$. More precisely, we will prove the following theorem:

\begin{theorem}
\label{th:norm_map}For any embedding of fields $j\colon F_1\emb F_2$ one can define the canonical map $$N_{F_2/F_1}\colon \SRL(F_1)\to \SRL(F_2)$$ satisfying the following properties:
\begin{enumerate}
\item $\SRL(j)\circ N_{F_2/F_1}=id$. \item If $F_1\subset F_2\subset F_3$ is a tower of extension from $\F_1$ then $N_{F_3/F_1}=N_{F_3/F_2}\circ N_{F_2/F_1}$.
\end{enumerate}
\end{theorem}

Item $(i)$ shows that $N_{F_2/F_1}$ is indeed an extension, while item $(ii)$ shows that this extension is functorial.

Let us explain the proof of Theorem \ref{th:strong_suslin_reciprocity_law_field}. The existence follows from Theorem \ref{th:norm_map} together with the fact that the element $N_{F/\k(t)}(\mathcal H_{\k(t)})$ does not depend on the embedding $j\colon \k(t)\emb F$. The uniqueness essentially follows from Galois descent for Milnor $K$-theory. %The Theorem \ref{th:two_dimensional_reciprocity_law} essentially follows from the proof of Theorem \ref{th:strong_suslin_reciprocity_law_field}.

Let us outline the proof of Theorem \ref{th:norm_map}. The proof of this theorem is in many respects similar to the construction of the norm map on Milnor $K$-theory. (See \cite{bass1973milnor,suslin1979reciprocity,MILNOR1969/70, kato1980generalization}). That is the reason why we denote it by the letter $N$. (Note that compared to the norm map in Milnor $K$-theory, in our case, the norm map is directed in the opposite direction. The reason for this is that while Milnor $K$-theory gives a covariant functor, the functor $\SRL$ is contravariant.) 

We define the map $N_{F_2/F_1}$ in two steps. First, for any discrete valuation $\nu$ of the field $F(t)$ we define the canonical map
$$\N_\nu\colon \SRL(F)\to \SRL(\overline {F(t)}_\nu).$$
Using this map, for any extension $F_1\emb F_2$ with a generator $a$ we will define the norm map $N_{F_2/F_1, a}\colon \SRL(F_1)\to \SRL(F_2)$ (see Definition \ref{def:norm_map}). Using ideas from \cite{suslin1979reciprocity} we will show that this map does not depend on $a$ and will have finished the proof of Theorem \ref{th:norm_map}. 

 Let us give the outline of the construction of the map $\mc N_\nu$. Let $F\in \F_1$. It is useful to divide the discrete valuations of the field $F(t)$ into two classes, namely  the general valuations and the special ones (see Definition \ref{def:types_of_valuations}). For the special valuations the definition of the map $\N_\nu$ is straightforward. Let us explain how to define $\mc N_\nu(h)$ for general valuations. We want to define $\mc N_\nu$ so that an analog of Theorem \ref{th:two_dimensional_reciprocity_law} would hold. This means that for any $h\in\SRL(F)$ and any $b\in \L^{m+2}F(t)^\times$ the following formula should hold:
 $$\sum\limits_{\nu\in \dval(F(t))}\mc N_\nu(h)(b)=0.$$
It turns out that this formula, together with the definition of $\mc N_\nu$ for special valuations, determines $\mc N_\nu$ completely. 
 
 To give a formal definition, we need a notion of \emph{lift}. Let $\nu\in \dval(F(t))$ be a general valuation and $n,j\in\mathbb N$. A lift of the element $a\in \Gamma(\overline{F(t)}_\nu, n)_j$ is an element $b\in \Gamma(F(t), n+1)_{j+1}$, such that the tame-symbol $\ts_\nu(b)$ is equal to $a$ and the tame-symbol of $b$ at any other general valuation vanishes. The set of all lifts of the element $a$ is denoted by $\mathcal L(a)$. It is easy to show that for any $a\in \Gamma(\ol {F(t)}_\nu,n)_j$ the set $\mathcal L(a)$ is non-empty. Let $\nu\in \dval(F(t))$ be a general valuation, $h\in \SRL(F)$ and $a\in \Lambda^{m+1}\overline{F(t)}_\nu$. Choose some lift $b\in \mathcal L(a)$ and define the element $\N_\nu(h)(a)$ by the following formula:
$$\mathcal N_\nu(h)(a)=-\sum\limits_{\mu\in\dval(F(t))_{sp}}\mathcal N_{\mu}(h)(\ts_\mu(b)).$$
Here $\dval(F(t))_{sp}$ denotes the set of all special valuations. In this formula the lifted reciprocity maps $\N_{\mu}(h)$ are already defined because $\mu$ is special. It remains to show that this expression does not depend on the choice of $b$ and for fixed $h$ gives a lifted reciprocity map on the field $\overline{F(t)}_\nu$. This can be done using certain properties of the lift and some version of the Parshin reciprocity law.

%Theorem \ref{th:strong_suslin_reciprocity_law_varities} follows directly from Theorem \ref{th:strong_suslin_reciprocity_law_field} and Theorem \ref{th:two_dimensional_reciprocity_law}.

\subsection{Case of of the field $\k(t)$}
\label{subsec:SRL_1}
Let us recall the following classical statement:
\begin{theorem}[Weil reciprocity law]
\label{theorem:Weil_rec_law}
    Let $X$ be a smooth projective curve over $\k$ and $f,g\in\k(X)^\times$. We have
    $$\prod_{x\in X(\k)}\ts_x(f\wedge g)=1.$$
\end{theorem}
We need the following lemma:
\begin{lemma}
\label{lemma:Res_is_zero}
Let $a\in \Gamma(\k(t),m+1)_{m}$. If $\delta_{m+1}(a)$ lies in the image of the multiplication map $\L^2 \k(t)^\t\otimes \L^{m-1}\k^\t\to\L^{m+1}\k(t)^\t$, then we have:
$$\sum\limits_{\nu\in\val(\k(t))}\ts_\nu(a)=0\in \im(\delta_m).$$
\end{lemma}
\begin{proof}
\begin{enumerate}
    \item Let $b=\delta_{m+1}(a)$. The elements of the form $$\{(t-c_1)/(t-c_2)\}_2\otimes c_3\wedge\dots\wedge c_{m+1}$$ lie in the kernel of the map $Tot_{\k(t)}$. Subtracting from $a$ a linear combination of  elements of these form, we can assume that $b$ lies in the image of the multiplication map $\k(t)^\t\otimes \L^{m}\k^\t\to\L^{m+1}\k(t)^\t$.
    
    \item Let $a''=(-1)^m\ts_{\infty}(a\wedge (1/t))$. For $x\in\k$ denote $a'_x= (-1)^{m+1}\ts_{x}(a)\wedge (t-x)$. We have $a'',a'_x\in \Gamma(\k(t),m+1)_m$. 
    
    Since the total residue of the elements $a'',a'_x$ is equal to zero, it is enough to prove the statement for the element $$\wt a=a-a''-\sum\limits_{x\in \k}a'_x.$$
    \item Let $\wt b=\delta_{m+1}(\wt a)$.  We claim that for any $x$ we have $\ts_{x}(\wt b)=\ts_{\infty}(\wt b\wedge (1/t))=0$. Indeed, we have 
    \begin{align*}
        &\ts_x(\wt b)=\ts_x(\delta_{m+1}(\wt a))=-\delta_{m}(\ts_x(\wt a))=-\delta_{m}(\ts_x(a-a''-\sum_x a_x'))=\\&-\delta_{m}(\ts_x (a)-\ts_x(a_x'))=-\delta_{m}(\ts_x(a)-\ts_x(a))=0.
    \end{align*}

    The proof of the formula $\ts_{\infty}(\wt b\wedge (1/t))=0$ is similar.
    
    On the other hand, we can write the element $\wt b$ in the following form
    $$ \wt b= \wt b'+\sum\limits_{x\in\k}\wt b_x\wedge (t-x),$$
    with $\wt b'\in \L^{m+1}\k^\times, \wt b_x\in \L^m\k^\times$.
    As $\wt b_x=(-1)^m\ts_x(\wt b), \wt b'=(-1)^{m+1}\ts_{\infty}(\wt b\wedge (1/t))$, it follows that $\wt b=0$.

    \item So we can assume that $\delta_{m+1}(a)=0$. In this case the statement follows from \cite[Corollary 1.4]{rudenko2021strong}. We remark that Corollary 1.4 was stated in loc.cit. only for $\C$. However this corollary was deduced from Theorem 1.2. which holds for any field. This implies that Corollary 1.4. holds for any algebraically closed field.)
\end{enumerate}
\end{proof}

\begin{proposition}
\label{prop:strong_low_P^1}
\label{prop:SRL_P1}
On the field $\k(t)$ there is a unique lifted reciprocity map. 
\end{proposition}

We will denote this lifted reciprocity map by $\mc H_{\k(t)}$.

\begin{proof}Denote by $A_1\subset \Lambda^{m+1}\k(t)^\t$ the image of $\delta_{m+1}$ and by $A_2$ the image of the multiplication map $\Lambda^2 \k(t)^\t\otimes \Lambda^{m-1}\k^\t\to\Lambda^{m+1}\k(t)^\t$. Elementary calculation shows that $A_1$ and $A_2$ together generate $\Lambda^{m+1}\k(t)^\t$. As any lifted reciprocity map is uniquely determined on $A_1$ and on $A_2$, uniqueness follows.

To show existence, define a map $$\mc H_{\k(t)}\colon \Lambda^{m+1}\k(t)^\t\to (B_2(\k)\otimes \Lambda^{m-2}\k^\t)/\im(\delta_m)$$ as follows. Let $a=a_1+a_2$, where $a_1\in A_1, a_2\in A_2$. Choose some $b\in \delta_{m+1}^{-1}(a_1)$ and define $\mc H_{\k(t)}(a)=\sum\limits_{x\in\mathbb P^1}\ts_x(b)$. This map is well-defined by Lemma \ref{lemma:Res_is_zero}. 

We already know that $\mc H_{\k(t)}$ satisfies the second property of Definition \ref{def:SRL}  and that the upper-left triangle of diagram (\ref{diagram:SRL_def}) is commutative. Let us show the bottom-right triangle is commutative. On $A_2$ its commutativity follows from Weil reciprocity law. Its commutativity on $A_1$ follows from the fact that the total residue map $Tot_F$ is a morphism of complexes.
\end{proof}

\begin{corollary}
\label{cor:SRL_on_Totaro}
    $$\mc H_{\k(t)}(t\wedge (1-t)\wedge (1-a/t)\wedge c_4\wdw c_{m+1})=-\{a\}_2\otimes c_4\wdw c_{m+1}.$$
\end{corollary}
\begin{proof}
Let $$b=\{t\}_2\otimes (1-a/t)\wedge c_4\wdw c_{m+1}.$$
We have:
\begin{align*}
    &Tot_{k(t)}(b)=-\{a\}_2\otimes c_4\wdw c_{m+1},\\& \delta_{m+1}(b)=t\wedge (1-t)\wedge (1-a/t)\wedge c_4\wdw c_{m+1}.
\end{align*}
As the map $\mc H_{\k(t)}$ is a lifted reciprocity map, we have
$$\mc H_{\k(t)}\delta_{m+1}(b)=Tot_{\k(t)}(b).$$ 
The statement follows.
\end{proof}

\subsection{The construction of the lift}
\label{sec:prel_results:lift}
In the definition of norm map on Milnor $K$-theory the point $\infty\in\P^1$ plays a special role. Let us call a discrete valuation on the field $\k(X)$ \emph{ special} if it corresponds to the point $\infty$ and \emph{general} otherwise. The following definition gives a two-dimensional analog of this definition.
\begin{definition}
\label{def:types_of_valuations}
Let $F\in \F_1$. A valuation $\nu\in \dval(F(t))$ is called \emph{general} if it corresponds to some irreducible polynomial over $F$. The set of general valuations are in bijection with  the set of all closed points on the affine line over $F$, which we denote by $\mb A_{F,(0)}^1$.  A valuation is called \emph{special} if it is not general. Denote the set of general (resp. special) valuations by $\dval(F(t))_{gen}$ (resp. $\dval(F(t))_{sp}$). 
\end{definition}

\begin{remark}
\label{rem:types_of_valuations}

Let $F\in\F_1$. Let us realize $F$ as a field of fractions on some smooth projective curve $X$ over $k$. Set $S=X\t \mathbb P^1$. It can be checked that a valuation $\nu\in \dval(F(t))$ is special in the following two cases:

\begin{enumerate}
    \item There is a birational morphism $\varphi\colon \wt S\to S$, and the valuation $\nu$ corresponds to some prime divisor $D\subset \wt S$ contracted under $\varphi$.
    \item The valuation $\nu$ corresponds to some of the divisors $X\t\{\infty\}, \{a\}\t\mathbb P^1, a\in X$.
\end{enumerate}

Otherwise, the valuation $\nu$ is general.
It follows from this description that if $\nu$ is a special valuation different from $X\times \{\infty\}$, then the residue field $\overline{F(t)}_\nu$ is isomorphic to $\k(t)$.
\end{remark}

\begin{definition}
\label{def:lift}
Let $F\in \F_1$, $j, m\in \mathbb N$, and $\nu\in \dval(F(t))_{gen}$. \emph{A lift} of an element $a\in \Gamma(\ol {F(t)}_\nu, m)_j$ is an element $b\in \Gamma(F(t), m+1)_{j+1}$ satisfying the following two properties:
\begin{enumerate}
    \item $\ts_\nu(b)=a$ and
    \item for any general valuation $\nu'\in\dval(F(t))_{gen}$ different from $\nu$, we have $\ts_{\nu'}(b)=0$.
\end{enumerate}

The set of all lifts of the element $a$ is denoted by $\mathcal L(a)$.
\end{definition}

\begin{theorem}
\label{th:main_exact_sequence}
Let $F\in \F_1$ and $j\in\{1,\dots, m\}$. For any $\nu\in \dval(F(t))_{gen}$ and $a\in \Gamma(\ol{F(t)}_\nu, m+1)_j$, the following statements hold:
\begin{enumerate}
    \item  The set $\mathcal L(a)$ is non-empty.
    \item Let us assume that $j=m+1$. For any $b_1,b_2\in \mathcal L(a)$, the element $b_1-b_2$ can be represented in the form $a_1+\delta_{m+2}(a_2)$, where $a_1\in \Lambda^{m+2} F^\t$ and $a_2\in \Gamma(F(t),m+2)_{m+1}$ such that for any $\nu\in\dval(F(t))_{gen}$ the element $\ts_\nu(a_2)$ lies in the image of the map $\delta_{m+1}$.
\end{enumerate}
\end{theorem}

In the case $m=2$, this theorem was proved in \cite[Theorem 2.6]{bolbachan_2023_chow}. This proof relied on Lemma 2.8 and Proposition 2.9 from \cite{bolbachan_2023_chow}. Both of these statements were formulated for arbitrary $m$. The proof of Theorem 2.6. from \cite{bolbachan_2023_chow} does not use any specific properties of the case $m=2$ and can be easily generalised to arbitrary $m$.

\subsection{Parshin reciprocity law}
\label{sec:prel_results:Parshin}

We need the following statement:

\begin{theorem}
\label{th:Parshin_sum_point_curve}
Let $L\in\F_2$ and $j\in\{m+1, m+2\}$. For any $b\in \Gamma(L,m+2)_j$ and all but finitely many $\mu\in \dval(L)$ the following sum is zero:
$$\sum\limits_{\mu'\in \val(\ol L_\mu)}\ts_{\mu'}\ts_\mu(b)=0.$$
Moreover, the following sum is zero:
$$\sum\limits_{\mu\in \dval(L)}\sum\limits_{\mu'\in \val(\oL_\mu)}\ts_{\mu'}\ts_\mu(b)=0.$$
\end{theorem}

The proof of this theorem is completely similar to the proof of Theorem 2.10 from \cite{bolbachan_2023_chow}. We remark that the proof of Lemma 2.14 from \cite{bolbachan_2023_chow} is not correct. However this lemma is a particular case of Lemma \ref{lemma:characterisation_of_strictly_regular_elements} from this paper.  See also Theorem \ref{th:Parshin_reciprocity_law} from this paper. 

%\begin{remark}
%    The proof of Theorem 2.10 from loc.cit relies on Lemma 2.14. The proof of this lemma is not correct. However this lemma is a particular case of Lemma \ref{lemma:characterisation_of_strictly_regular_elements} from this paper.
%\end{remark}

We recall that the following lemma was stated in Section \ref{sub:sec:vanishing_statements}

\begin{lemma}[Lemma \ref{lemma:finitnes_of_sum}]
Let $S$ be a smooth proper surface. Let $b\in \Lambda^{m+2} L^\t$ be strictly regular. For any alteration $\ph\colon \wt S\to S$ and any divisor $E\subset \wt S$ contracted under $\ph$, the element $\ts_E(\ph^*(b))$ lies in the image of the multiplication map $\overline L_\nu^\t\otimes \Lambda^m \k^\t\to \Lambda^{m+1}\overline L_\nu^\t$.
\end{lemma}

\begin{corollary}
\label{cor:finitnes_of_sum}
    Let $L\in\F_2$. For any $b\in \Lambda^{m+2} L^\t$ and all but a finite number of $\nu\in \dval(L)$ the element $\ts_\nu(b)$ belongs to the image of the multiplication map $\overline L_\nu^\t\otimes \Lambda^m \k^\t\to \Lambda^{m+1}\overline L_\nu^\t$.
\end{corollary}

\subsection{Definition of the map $\mc N_\nu$}

Let $F\in \F_1$, $\nu\in\dval(F(t))$. Denote the field $F(t)$ by $L$. Our goal is to construct a map $\N_\nu\colon \SRL(F)\to \SRL(\oL_\nu)$. We will do this in the following three steps:

\begin{enumerate}
    \item We will define this map when $\nu$ is a special valuation.
    \item Using the construction of the lift from Section \ref{sec:prel_results:lift}, for any general valuation $\nu\in\dval(L)$, we will define a map $$\N_\nu\colon \SRL(F)\to \Hom(\Gamma(\ol L_\nu, m+1)_{m+1}\to \Gamma(\k,m)_{m-1}/\im(\delta_m)).$$
    \item Using Theorem \ref{th:Parshin_sum_point_curve} from Section \ref{sec:prel_results:Parshin}, we will show that for any $h\in\SRL(F)$, the map $\N_\nu(h)$ is a lifted reciprocity map on the field $\oL_\nu$. So, $\N_\nu$ gives a map $\SRL(F)\to \SRL(\ol{L}_\nu)$.
\end{enumerate}

Denote by $\nu_{\infty, F}\in\dval(F(t))$ a discrete valuation corresponding to the point $\infty\in\P^1_F$. Let $\nu$ be special. If $\nu=\nu_{\infty,F}$ then define $\N_\nu(h)=h$ (here we have used the identification of $\oL_{\nu_{\infty,F}}$ with $F$). In the other case, we have $\overline{L}_\nu\simeq \k(t)$ (see Remark \ref{rem:types_of_valuations}). In this case, define $\N_\nu(h)$ to be the unique lifted reciprocity map from Proposition \ref{prop:SRL_P1}. We have defined $\N_\nu$ for any $\nu\in \dval(L)_{sp}$. 

We recall that the set $\mc L(a)$ was introduced in Definition \ref{def:lift} and consists of different lifts of the element $a$. Let $h\in\SRL(F)$. Define a map $H_h\colon \Lambda^{m+2} L^\t\to \Gamma(\k,m)_{m-1}/\im(\delta_m)$ by the following formula:
$$H_h(b)=-\sum\limits_{\mu\in \dval(L)_{sp}}\N_\mu(h)(\ts_{\mu}(b)).$$ 
This sum is well-defined by Corollary \ref{cor:finitnes_of_sum}. 
\begin{definition}
\label{def:theta_gen}
Let $\nu\in\dval(L)_{gen}$. Define a map $$\N_\nu\colon \SRL(F)\to \Hom(\Lambda^{m+1}\ol L_\nu^\t,  \Gamma(\k,m)_{m-1}/\im(\delta_m))$$ as follows. Let $h\in\SRL(F)$ and $a\in \Lambda^{m+1}\ol L_\nu^\t$. Choose some lift $b\in\mc L(a)$ and define the element $\N_\nu(h)(a)$ by the formula $H_h(b)$. 
\end{definition}

Similarly to Section 3.1 of \cite{bolbachan_2023_chow}, it can be shown that this definition is well-defined and gives a map $\N_\nu\colon \SRL(F)\to \SRL(\ol L_\nu).$ We omit details.

\subsection{Norm map}

\begin{definition}
\label{def:norm_map}
Let $j\colon F_1\emb F_2$ be an extension of some fields from $\F_1$. Let $a$ be some generator of $F_2$ over $F_1$. Denote by $p_a\in F[t]$ the minimal polynomial of $a$ over $F_1$. Denote by $\nu_a$ the corresponding valuation. The residue field $\overline{L}_{\nu_{a}}$ is canonically isomorphic to $F_2$. So we get a map $\N_{{\nu_a}}\colon\SRL(F_1)\to \SRL(F_2)$, which we denote by $N_{F_2/F_1, a}$. This map is called \emph{the norm map}. 
\end{definition}

Then one can prove the following theorem:

\begin{theorem}
    The map $N_{F_2/F_1, a}$ does not depend on $a$. Denote it simply by $N_{F_2/F_1}$. We have: \begin{enumerate}
        \item If $F_1\subset F_2\subset F_3$ are extensions from $\F_1$, then $$N_{F_3/F_2}\circ N_{F_2/F_1}=N_{F_3/F_1}.$$
        \item Let $j\colon F_1\emb F_2$. We have $\SRL(j)\circ N_{F_2/F_1}=id$.
        \item Let $F\in\F_1$. Choose some embedding $j\colon \k(t)\emb F$. The element $$N_{F/\k(t)}(h_{\k(t)})$$ does not depend on $j$.
    \end{enumerate}
\end{theorem}

The proof of this theorem is similar to the construction of the norm map on Milnor $K$-theory \cite{suslin1979reciprocity}. Case $m=2$ was proved in \cite[Theorem 1.14]{bolbachan_2023_chow}. The general case is similar.

Let us explain the proof of Theorem \ref{th:strong_suslin_reciprocity_law_field}. Item 1 follows directly from the previous theorem. (See the proof of \cite[Theorem 1.8.]{bolbachan_2023_chow}). Item 2 follows from Corollary \ref{cor:SRL_on_Totaro} and Proposition \ref{prop:strong_low_P^1}.

The proof of the first item of Theorem \ref{th:two_dimensional_reciprocity_law} is similar to the proof of \cite[Corollary 1.13]{bolbachan_2023_chow}. The second item follows from Lemma \ref{lemma:finitnes_of_sum}.

\section{The proof of the main result}
\label{sec:Lambda_and_polylogarithms}

\subsection{The map $\mc T$}
\label{sec:Totaro_map}
We recall that the chain complex $\Gamma_{\geq m-1}(\k,m)$ is concentrated in degrees $m-1,m$ and has the following form:
$$\B_2(\k)\otimes\Lambda^{m-2}(\k^\times\otimes_\Z\Q)/S\xrightarrow{\delta_m} \Lambda^m(\k^\times\otimes_\Z\Q).$$
The subgroup $S$ is generated by the elements of the form $$\{a\}_2\otimes a\wedge c_4\wdw c_m.$$ 
The map $\delta_m$ is defined by the formula $\delta_m(\{a\}_2\otimes c_3\wdw c_m)=a\wedge(1-a)\wedge c_3\wdw c_m$.

Let us recall the definition of the map $\T$. It is a map from $\wG(\k,m)$ to $\wCH(\k,m)$, where $\wCH(\k,m)=\tau_{\geq m-1}(\CH(\k,m))/\M(\k,m)$. The element $\{a\}_2\otimes c_3\wdw c_{m}$ goes to 
$[t,(1-t),(1-a/t),c_3,\dots, c_{m}]_{\P^1}.$ The element $c_1\wdw c_m$ goes to
$[c_1,\dots, c_m]_{\spec\k}.$
The goal of this subsection is to prove the following proposition. 

\begin{proposition}
\label{prop:T_is_well_defined}
    The map $\T$ is a well-defined morphism of complexes.
\end{proposition}

The proof of this proposition essentially repeats arguments from \cite{cubical1999herbert}. %However, we believe that the computations in the complex $\Lambda$ have some advantages over the complex calculating higher Chow groups. Therefore, we decided to include this proof to support the above statement with a concrete example.

%\section{A direct proof that $\T'$ is well-defined}
%\label{sec:Totaro_is_well_defined}
%From the definition it follows that $\T'$ is the composition $\wW\circ \T$. In particular $\T'$ is a morphism of complexes. However, it is interesting to prove this fact directly, without using the results of \cite{cubical1999herbert}. 
Let $a\in \k\bs\{0\}$. Define
$$T_a = [\P^1, (t\wedge (1-t)\wedge (1-a/t))].$$
Let $V$ be a $2$-dimensional vector space over $\k$ and $l_i\in V^{*}$. Denote by $T(l_1, l_2, l_3, l_4)$ the element $[\P(V), \omega(l_1, \dots, l_4)]$, where $$\omega(l_1, l_2, l_3, l_4)=\dfrac {l_1}{l_4}\wedge \dfrac {l_2}{l_4}\wedge \dfrac {l_3}{l_4}.$$

We recall that cross-ratio is defined by the following formula 

$$c.r.(a,b,c,d)=\dfrac{(a-c)(b-d)}{(a-d)(b-c)}.$$

We need the following lemma:

\begin{lemma}
\label{lemma:Abel_five_term_relations}
    Let $x_1,\dots, x_5$ be five different points on $\mb \P^1$. Then
    $$\sum\limits_{i=1}^5(-1)^iT_{c.r.(x_1,\dots,\hat{x_i},\dots, x_5)}=0\in\L(\k,2)_{1}.$$
\end{lemma}

This lemma, in turn, follows from another lemma:

\begin{lemma}The following statements are true:
    \begin{enumerate}
        \item Let $l_1,\dots, l_4\in V^*$. Assume that any two of these vectors are linearly independent. Then we have
        $$T(l_1, l_2, l_3, l_4)=-T_{c.r.(\pi(l_1),\pi(l_2), \pi(l_3), \pi(l_4))}.$$
        In this formula $\pi\colon V^*\bs \{0\}\to \P(V^*)$ is the natural projection and $c.r.(\cdot)$ is the cross-ratio.
         \item Let $l_1,\dots, l_5\in V^*$. Assume that any two of these vectors are linearly independent. Then we have
        $$\sum\limits_{i=1}^5(-1)^iT(l_1,\dots\hat{l_i},\dots l_5)=0.$$
    \end{enumerate}
\end{lemma}

\begin{proof}\begin{enumerate}
\item
   It follows from Lemma \ref{lemma:Beilinson_Soule_vanishing} that for any $\lambda_1,\lambda_2, \lambda_3, \lambda_4\in \k^\t$ we have
    $$T(l_1, l_2, l_3, l_4)=T(\lambda_1 l_1, \lambda_2 l_2,\lambda_3 l_3,\lambda_4 l_4).$$
    So we can assume that $V=\k^2$ and $l_1=(1,-a), l_2=(1,-b), l_3=(1,-c), l_4=(0,1)$ for some $a,b,c\in \k$. We get:
    $$T(l_1, l_2, l_3, l_4)=[\P^1, (t-a)\wedge(t-b)\wedge (t-c)].$$

    Let $\varphi\cl \P^1\to\P^1$ be a map given by rhe formula $\varphi(t)=t(b-c)+c$. We get:
    \begin{align*}
       & [\P^1, (t-a)\wedge (t-b)\wedge (t-c)]=(\deg \varphi^{-1})[\P^1, \varphi^*((t-a)\wedge (t-b)\wedge (t-c))]=\\
       &[\P^1, ((b-c)t-(a-c))\wedge ((b-c)t-(b-c))\wedge ((b-c)t)].
    \end{align*}

    By lemma \ref{lemma:Beilinson_Soule_vanishing} we know that for any $f,g\in\k(t)^\times$ and $c\in\k^\times$ we have $[\P^1,f\wedge g\wedge c]=0$. We get
    \begin{align*}
        &[\P^1, ((b-c)t-(a-c))\wedge ((b-c)t-(b-c))\wedge ((b-c)t)]=\\&[\P^1, (t-(a-c)/(b-c))\wedge (1-t)\wedge t]=-T_{(a-c)/(b-c)}.
    \end{align*}

On the other hand
$$c.r.(\pi(l_1), \pi(l_2), \pi(l_3), \pi(l_4))=c.r.(-1/a,-1/b,-1/c,0)=c.r.(a,b,c,\infty)=\dfrac{a-c}{b-c}.$$

\item This formula follows from the following formula which is obtained by a direct computation:
$$\sum\limits_{i=1}^5(-1)^i\omega(l_1,\dots, \widehat{l_i},\dots, l_5)=0\in \EP{3}{\P(V)}.$$
\end{enumerate}
\end{proof}

\begin{proof}[The proof of Proposition \ref{prop:T_is_well_defined}]
    According to \cite{goncharov1994polylogarithms}, the group $\mc R_2(\k)$ is generated by the following elements:
    $$\sum\limits_{i=1}^5(-1)^i\{c.r.(z_1,\dots, \widehat z_i,\dots, z_5)\}_2, \{0\}_2,, \{1\}_2, \{\infty\}_2.$$
    In this formula $z_i$ are five different points on $\P^1$. Let $b_1,b_2,c_4,\dots, c_{n}$ are non-zero elements of $\k$. We need to show that in $\wCH(\k,m)$ the following relations hold:
    \begin{enumerate}
        \item $\sum\limits_{i=1}^5(-1)^i[t,1-t,1-c.r.(z_1,\dots, \hat z_i,\dots, z_5)/t, c_4,\dots, c_{n}]_{\P^1}=0.$
        \item We have
        \begin{align*}
            &[t,1-t,1-a/t, b_1b_2,c_5,\dots, c_{n}]_{\P^1}=\\&[t,1-t,1-a/t, b_1,c_5,\dots, c_{n}]_{\P^1}+[t,1-t,1-a/t,  b_2,c_5,\dots, c_{n}]_{\P^1}.
            \end{align*}
        \item $[t,1-t,1-a/t, a, c_5,\dots, c_{n}]_{\P^1}=0.$
        \item $d([t,1-t,1-a/t, c_4, c_5,\dots, c_{n}]_{\P^1})=[a,1-a,c_4,\dots, c_{n}]_{\P^1}.$
    \end{enumerate}
    Here is the proof:
    \begin{enumerate}
        \item We can assume that $m=2$. Denote by $\delta_2$ the differential in the complex $\G(\k,2)$. It is well-known(and follows from the fact that $\delta_2$ is well-defined) that $\delta_2$ vanishes on the Abel five-term relation. As the group $\wCH(\k,2)_2$ can be identified with $\L^2\k^\t$, this implies that the element
        $$\sum\limits_{i=1}^5(-1)^i[t,1-t,1-c.r.(z_1,\dots, \hat z_i,\dots, z_5)/t]_{\P^1}$$
        is closed.
        We need to show that this element is equal to zero. By Lemma \ref{lemma:Galois_descent} the cohomology $H^i(\wCH(\k,m))$ satisfies Galois descent. So we can assume that $\k$ is algebraically closed. In Section \ref{sec:W_isomorphsim} we showed that the map $\wW\colon \wCH(\k,m)\to\wL(\k,m)$ is an isomorphism. So it is enough to prove the corresponding relation in the complex $\wL(\k,m)$. This follows from Lemma \ref{lemma:Abel_five_term_relations}. See also \cite{cubical1999herbert}.
        \item Let $(t,w)$ be coordinates in $\P^1\t\P^1$. The expression is the boundary of the following cycle:
        $$\left[t, 1-t, (1-a/t), \dfrac{(w-b_1b_2)(w-1)}{(w-b_1)(w-b_2)}, w, c_5,\dots, c_n\right]_{\PP}.$$
        \item It is enough to prove that
        \begin{align*}
        &d([x_1,x_2,(1-x_1),(1-x_2/x_1),(1-a/x_2),\\& c_5,\dots, c_{n}]_{\P^2})=
        [t,(1-t),(1-a/t),a,c_5,\dots,c_{n}]_{\P^1}.
    \end{align*}
This formula was proved in \cite{bloch_rriz_1994_mixed}, see also \cite{goncharov_levin_gangl_2009_muly_polyl_alg_cycles}. Let $f_i$ be the $i$-th function on the left hand-side. Denote by $Z$ the closure of the zeros of the functions $f_i-1$. Let $U=\P^2\bs Z$. It is easy to see that the only prime components of the divisors of $f_i$ on $U$ are $x_1=x_2$ and $x_2=a$. Moreover these divisors belong only to the functions $1-x_2/x_1$ and $1-a/x_2$ correspondingly. This imply that the cycle on the left hand-side is admissible. For the divisor $x_1=x_2$ the term in the differential is zero. For  the divisor $x_2=a$ we get 
$$[t,1-t,1-a/t, a, c_5,\dots, c_{n}]_{\P^1}.$$
        \item Indeed the cycle intersects only the face given by the equation $x_3=0$. We get $t=a$ and so we get the cycle $[a,1-a,c_3,\dots, c_n]$. The statement follows.
        \end{enumerate}
\end{proof}

\subsection{The map $\mc T'$}
\label{subsec:map:Totaro_prime}
Define a morphism of complexes $$\T'\colon \wG(\k,m)\to \wL(\k, m)$$ as follows. The element $\{a\}_2\wedge c_3\wdw c_m$ goes to
$$[\P^1, t\wedge (1-t)\wedge (1-a/t)\wedge c_3\wdw c_m].$$
The element $c_1\wdw c_m$ goes to
$[\spec \k, c_1\wdw c_m]$. It follows from the definition that $\mc T'=\wW\circ\mc T$. In particular $\mc T'$ is a well-defined morphism of complexes. 

\begin{theorem}[Theorem \ref{th:Totaro_isomorphism_intro}]
\label{th:Totaro_is_morphism_of_complexes}
    Let $\k$ be algebraically closed. The map $\T'$ is an isomorphism.
\end{theorem}

Let us show that the map $\mc T'$ is surjective. This follows from the following result:

\begin{proposition}
\label{prop:Lambda_generated_by_Totaro_cycles}
    Let $\k$ be algebraically closed. The group $\L(\k, m)_{m-1}/\im(d)$ is generated by the elements of the form
    $$[\P^1, t\wedge (1-t)\wedge (1-a/t)\wedge c_4\wdw c_{n}].$$   
\end{proposition}

\begin{lemma}
\label{lemma:P1_genrated_by_Totaro}
    Let $S=\P^1\t\P^1$. Consider the element
    $$a=[S, x\wedge (1-x)\wedge \alpha_2\wdw \alpha_k\wedge c_{k+1}\wdw c_n].$$
    In this formula $x$ is the canonical coordinates on the first $\P^1$, $\alpha_i\in\k(\P^1\t\P^1)$ and $c_i\in \k$. For any birational morphism $\ph\colon \wt S\to S$, and any divisor $E$ contracted under $\ph$, the element $[E, \ts_E(\ph^*(a))]$ can be represented as linear combination of the elements of the form
    $$[\P^1, c\wedge (1-c)\wedge  \beta_3\wdw \beta_{k} \wedge c_{k+1}\wdw c_n].$$
    In this formula $c\in \k\bs\{0,1\}$ and $\beta_i\in \k(\P^1)$.
\end{lemma}

\begin{proof}[The proof of Lemma \ref{lemma:P1_genrated_by_Totaro}]
    Denote by $(x,t)$ the canonical coordinates on $S$. Let $\{(x_0,t_0)\}\subset S$ be the image of $E$. If $x_0=0$ then the the restriction of the function $\ph^*(1-x)$  to $E$ is equal to $1$ and so $[E, \ts_E(\ph^*(a))]=0$. The case $x_0=1$ is similar. Let us assume that $x_0=\infty$. In this case, the restriction of the function $\ph^*(x/(1-x))$ to $E$ is equal to $-1$. As the element $a$ can be rewritten in the form 
    $$a=[S, (x/(1-x))\wedge (1-x)\wedge \alpha_2\wdw \alpha_k\wedge c_{k+1}\wdw c_n],$$
    this implies that $[E, \ts_E(\ph^*(a))]=0.$
    
    So we can assume that $x_0\in \k\bs\{0,1\}$. This implies that the restriction of the function $\ph^*(x)$ to $E$ is constant. Denote this constant by $c$. We get
    $$[E, \ts_E(\ph^*(a))]=[E, c\wedge (1-c)\wedge \ts_{E}(\alpha_2\wdw \alpha_k)\wedge c_{k+1}\wdw c_n].$$
    The statement follows.
\end{proof}

\begin{proof}[The proof of Proposition \ref{prop:Lambda_generated_by_Totaro_cycles}]
Denote by $A$ the subgroup of the group $\L(\k, m)_{m-1}/\im(d)$ generated by the elements stated in the lemma.

    By Theorem \ref{th:Lambda_is_generated_by_rationally_connected} we know that the group $\L(\k, m)_{m-1}/\im(d)$ is generated by the elements of the form
    $$[\P^1, (t-a_1)\wdw (t-a_k)\wedge c_{k+1}\wdw c_n].$$
    Let us prove by induction on $k, k\geq 0$, that this element lies in $A$. 
    \begin{enumerate}
        \item The case $k\leq 2$ follows from Lemma \ref{lemma:Beilinson_Soule_vanishing}.
        \item Let $k=3$. Denote by $\ph$ an automorphism of $\P^1$ given by the formula $\ph(t)=a_1+(a_2-a_1)t$. We get
        \begin{align*}
            &[\P^1, (t-a_1)\wedge (t-a_2)\wedge (t-a_3)\wedge c_4\wdw c_n]=\\&[\P^1, \ph^*((t-a_1)\wedge (t-a_2)\wedge (t-a_3)\wedge c_4\wdw c_n)]=\\
            &[\P^1, (t(a_2-a_1))\wedge (t(a_2-a_1)-(a_2-a_1)\wedge (t(a_2-a_1)-(a_3-a_1))\wedge c_4\wdw c_n]=\\
            &\left[\P^1, t\wedge (1-t)\wedge \left(t-\dfrac {a_3-a_1}{a_2-a_1}\right)\wedge c_4\wdw c_n\right]=\\
            &\left[\P^1, t\wedge (1-t)\wedge \left(1-\left(\dfrac {a_3-a_1}{a_2-a_1}\right)/t\right)\wedge c_4\wdw c_n\right].
        \end{align*}
        \item Let $k\geq 4$. Consider the element
        \begin{align*}
            &[\P^1\t\P^1, x\wedge (1-x)\wedge (x-(t-a_1)/(t-a_2))\wedge\\
            &(t-a_3)\wdw (t-a_k)\wedge c_{k+1}\wdw c_n].
        \end{align*}
    
    It follows from Lemma \ref{lemma:P1_genrated_by_Totaro} that the differential of this element is equal to $x+y$, where $x\in A$ and 
    $$y=(t-a_1)/(t-a_2)\wedge (1-(t-a_1)/(t-a_2))\wedge (t-a_3)\wdw (t-a_{k})\wedge c_{k+1}\wdw c_n.$$
    By inductive assumption this implies that the element
    $$[\P^1, (t-a_1)\wdw (t-a_k)\wedge c_{k+1}\wdw c_n]$$
    lies in $A$.
    \end{enumerate}

\end{proof}

It remains to show that the map $\T'$ is injective. We recall that in the previous section for any smooth complete curve over $\k$ we constructed the canonical map $\mc{H}_{\k(X)}\cl \EP{m+1}{X}\to \wG(\k,m)_{m-1}$. Define a morphism of complexes 
$$\mc{CD}\colon \wL(\k,m)\to \wG(\k,m)$$
as follows. The image of the element $[X, a]\in \wL(\k,m)_{m-1}$ is equal to $-\mc H_{\k(X)}(a)$. The image of the element $[\spec \k, a]$ is equal to $a$. It follows from the results of the previous section that $\mc{CD}$ is a morphism of complexes and $\mc{CD}\circ \T'=id$. This implies that the map $\mc T'$ is injective. So the map $\T'$ is an isomorphism.

\subsection{The proof of Theorem \ref{th:main_second_intro}}In the beginning of this section we have shown that the map $\T\cl \wG(\k,m)\to\wCH(\k,m)$ is well-defined. By the result of \cite{rudenko2021strong} and Lemma \ref{lemma:Galois_descent}, the groups $$H^i(\wCH(\k,m)), H^i(\wG(\k,m))$$ satisfy Galois descent. So we can assume that $\k$ is algebraically closed. In this case, it is enough to show that $\T$ is an isomorphism. In Section \ref{sec:W_isomorphsim} we have constructed the map $\wW\cl \wCH(\k,m)\to \wL(\k,m)$ and have shown that this map is an isomorphism. So it remains to prove that the map $\T'=\W\circ\T$ is an isomorphism. This was proven in the previous subsection.

Vasily Bolbachan\\
Skolkovo Institute of Science and Technology, 121205, Russia, Moscow, Bolshoy Boulevard, 30, p.1;\\ HSE University, HSE-Skoltech International Laboratory of Representation
Theory and Mathematical Physics, 119048, Russia, Moscow, Usacheva str., 6\\
\emph{E-mail:} \texttt{vbolbachan{\fontfamily{ptm}\selectfont @}gmail.com}
\end{document}